\documentclass[final]{article}
\usepackage[T1]{fontenc}
\usepackage{amsmath,amsfonts,amsthm,mathrsfs,amssymb}
\usepackage{cite}
\usepackage{multirow}
\usepackage{algorithmic,algorithm}
\usepackage{graphicx,float}
\usepackage{subfigure}
\usepackage{placeins}
\usepackage{color}
\usepackage{indentfirst}
\usepackage{bm}
\usepackage[notref,notcite]{showkeys}
\usepackage[colorlinks,
            linkcolor=blue,
            anchorcolor=blue,
            citecolor=red
            ]{hyperref}
\usepackage[latin1]{inputenc}
\numberwithin{equation}{section}
\newtheorem{theorem}{Theorem}[section]
\newtheorem{lemma}[theorem]{Lemma}
\newtheorem{corollary}[theorem]{Corollary}

\newtheorem{remark}[theorem]{Remark}

\providecommand{\Div}{\operatorname{div}}          
\providecommand{\curl}{\operatorname{{\bf curl}}}  

\providecommand*{\Dist}[2]{\operatorname{dist}({#1};{#2})}   
\providecommand*{\Dist}[2]{\Dist{#1}{#2}}
\providecommand*{\Span}[1]{\operatorname{Span}\left\{{#1}\right\}}     
\providecommand{\Supp}{\operatorname{supp}}                            

\renewcommand{\Re}{\operatorname{Re}}             
\renewcommand{\Im}{\operatorname{Im}}             

\newcommand{\Vi}{{\mathbf{i}}}

\newcommand{\Bn}{{\boldsymbol{n}}}

\newcommand{\Bp}{{\boldsymbol{p}}}

\newcommand{\Br}{{\boldsymbol{r}}}

\newcommand{\Bv}{{\boldsymbol{v}}}

\newcommand{\Bz}{{\boldsymbol{z}}}

\newcommand{\BE}{{\boldsymbol{E}}}

\newcommand{\BH}{{\boldsymbol{H}}}

\newcommand{\thetabf}{\boldsymbol{\theta}}

\newcommand{\Ci}{\mathcal{I}}

\newcommand{\Ck}{\mathcal{K}}

\newcommand{\Cp}{\mathcal{P}}

\newcommand*{\N}[1]{\left\|{#1}\right\|}     
\newcommand*{\SN}[1]{\left|{#1}\right|}      

\newcommand*{\Lp}[2][\defaultdomain]{L^{#2}({#1})}

\newcommand*{\NLp}[3][\defaultdomain]{\N{#2}_{\Lp[#1]{#3}}}

\newcommand*{\Ltwo}[1][\defaultdomain]{\Lp[#1]{2}}

\newcommand*{\NLtwo}[2][\defaultdomain]{\NLp[#1]{#2}{2}}

\newcommand*{\Hm}[2][\defaultdomain]{H^{#2}({#1})}

\newcommand*{\bHm}[3][\defaultdomain]{H_{#3}^{#2}({#1})}

\newcommand*{\Hone}[1][\defaultdomain]{\Hm[#1]{1}}

\newcommand*{\zbHone}[1][\defaultdomain]{\bHm[#1]{1}{0}}

\newcommand*{\NHone}[2][\defaultdomain]{{\N{#2}}_{\Hone[{#1}]}}

\newcommand*{\Hdiv}[1][\defaultdomain]{\boldsymbol{H}(\Div,{#1})}

\newcommand*{\NHdiv}[2][\defaultdomain]{\N{#2}_{\Hdiv[#1]}}

\newcommand{\D}{\mathrm{d}}

\newcommand{\ol}{\overline}

\newcommand{\be}{\begin{eqnarray}}
\newcommand{\ee}{\end{eqnarray}}

\newcommand{\ben}{\begin{eqnarray*}}
\newcommand{\een}{\end{eqnarray*}}

\newcommand{\pml}{\mathrm{pml}}

\definecolor{rot}{rgb}{1.000,0.000,0.000}
\definecolor{rot1}{rgb}{0.000,0.000,1.000}

\begin{document}
\renewcommand{\theequation}{\arabic{section}.\arabic{equation}}

\begin{titlepage}
  \title{A PML method for signal-propagation problems in axon}

\author{
Xue Jiang\thanks{Department of Mathematics, Faculty of Science, Beijing University of Technology, Beijing 100124, China. Email:{\tt jxue@lsec.cc.ac.cn}},
Maohui Lyu\thanks{LSEC, Institute of Computational Mathematics and Scientific/Engineering Computing, Academy of Mathematics and Systems Science, Chinese Academy of Sciences, Beijing 100190, China. Email:{\tt mlyu@lsec.cc.ac.cn}},
Tao Yin\thanks{LSEC, Institute of Computational Mathematics and Scientific/Engineering Computing, Academy of Mathematics and Systems Science, Chinese Academy of Sciences, Beijing 100190, China. Email:{\tt yintao@lsec.cc.ac.cn}},
Weiying Zheng\thanks{LSEC, Institute of Computational Mathematics and Scientific/Engineering Computing, Academy of Mathematics and Systems Science, Chinese Academy of Sciences, Beijing 100190, China. Email:{\tt zwy@lsec.cc.ac.cn}}
}
\end{titlepage}
\maketitle
%
\begin{abstract}
This work is focused on the modelling of signal propagations in myelinated axons to characterize the functions of the myelin sheath in the neural structure. Based on reasonable assumptions on the medium properties, we derive a two-dimensional neural-signaling model in cylindrical coordinates from the time-harmonic Maxwell's equations. The well-posedness of model is established upon Dirichlet boundary conditions at the two ends of the neural structure and the radiative condition in the radial direction of the structure. Using the perfectly matched layer (PML) method, we truncate the unbounded background medium and propose an approximate problem on the truncated domain. The well-posedness of the PML problem and the exponential convergence of the approximate solution to the exact solution are established. Numerical experiments based on finite element discretization are presented to demonstrate the theoretical results and the efficiency of our methods to simulate the signal propagation in axons.
\end{abstract}

\section{Introduction}
\label{sec:intro}

The problem of the signal transmission in neural system is one of the most fundamental and important issues in neuroscience. Axons are the primary transmission lines of the nervous system and can be characterized into two types: myelinated and unmyelinated axons. It is studied in \cite{Metal14} that generation of new myelin is important for learning motor skills. For the myelinated axons, the myelin sheath is a layer of membrane wrapped around the axons and gaps in the myelin sheath, known as nodes of Ranvier, occur at evenly spaced intervals. It is understood that the functions of the myelin sheath and nodes of Ranvier are to insulate and cause the saltatory conduction of the action potential. However, it seems extremely difficult to experimentally observe the signal propagation in axon and quantitatively or qualitatively describe the effects of myelin sheath and nodes of Ranvier. Therefore, it has become more and more important to derive and investigate efficient mathematical and physical models to numerically simulate the transmission of signal in axon.

In open literature, the transmission of signals in axon is commonly treated by an equivalent circuit, see for example \cite{HH521}. However, the model of an equivalent circuit, wherein frequencies in the kHz range is considered, is not consistent with the fact that biological macromolecules usually exhibit collective vibrations in the electromagnetic field in the infrared to terahertz (THz) spectral range \cite{LHGM08,LD10,WL99} and furthermore, the roles of myelin features remain poorly understood. Recently, a novel dielectric waveguide model is proposed in \cite{L19} to explain the mechanism of infrared and terahertz neurotransmission through myelinated nerves. It is experimentally demonstrated in \cite{L19} that, at a certain THz/infrared frequency region, myelin exhibits a significantly higher refractive index than axons which supports the hypothesis that the myelin sheath serves as a dielectric waveguide. Then based on an electromagnetic waveguide model, the explicit waveguide modes can be calculated by assuming that an infinite axon is completely wrapped by the myelin sheath without any node of Ranvier. But this analytic method is not applicable for the case of finite myelinated axon with nodes of Ranvier, for which, only a schematic illustration has been provided. The numerical simulation of analogous waveguide models for myelinated axon has also been considered in \cite{ZMGT18,ZZML22}. However, these results still can not effectively characterize the functions of myelin sheath, and moreover, there is no mathematical and numerical analysis, for example, the well-posedness and convergence, for the considered model and associated numerical solver.

Inspired by \cite{L19,ZMGT18,ZZML22}, this paper devotes to proposing a novel waveguide problem of the electromagnetic waves to model the signal transmission in myelinated axon and providing solid mathematical analysis and numerical demonstration for the deduced model, for which the corresponding numerical analysis is left for future works. As shown in Figure~\ref{model}(a), the signal propagation in the integrated neuron is quite complicated. Given an electromagnetic signal at one end of the axon, the propagation interested in this work is only restricted to the axon region as well as the myelin sheath.  As shown in Figure~\ref{model}(b) in cylindrical coordinates, let $D_1, D_2$ denote the domain of axon and myelin sheath, respectively. Note that the skin depth is much larger than the neurological scale. Thus, exterior to the myelinated axon, an infinite domain $D^c$ of fluid is assumed, see Figure~\ref{model}(b). In addition, the cross section of myelinated axon can be viewed as a concentric structure. Then from the classical Maxwell's equations and assuming the wave fields to be always perpendicular to the direction of propagation (i.e., the length direction of axon) and independent of the angle variable, new time-harmonic TM and TE models in cylindrical coordinates, are introduced and appropriate boundary conditions are imposed at the end of axon $\Gamma_{\mathrm{left}}$, $\Gamma_{\mathrm{right}}$ for approximation. Given an incident field on $\Gamma_{\mathrm{left}}$, a zero mixed Dirichlet and Neumann boundary condition on $\Gamma_{\mathrm{right}}$ is considered. The discussion of the corresponding time-dependent model remains individually interesting and is left for future works.

\begin{figure}[htbp]
\centering
\begin{tabular}{cc}
\includegraphics[scale=0.3]{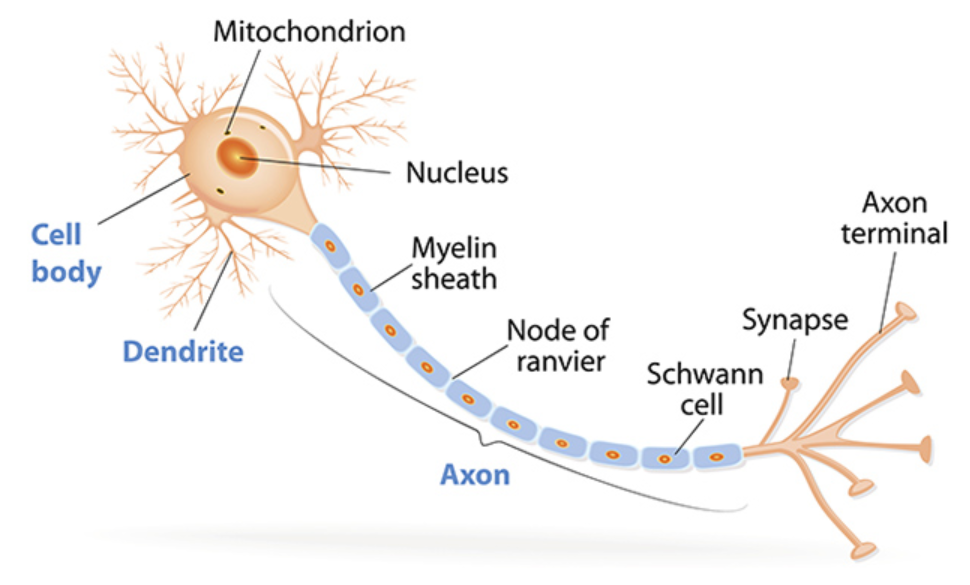} &
\includegraphics[scale=0.3]{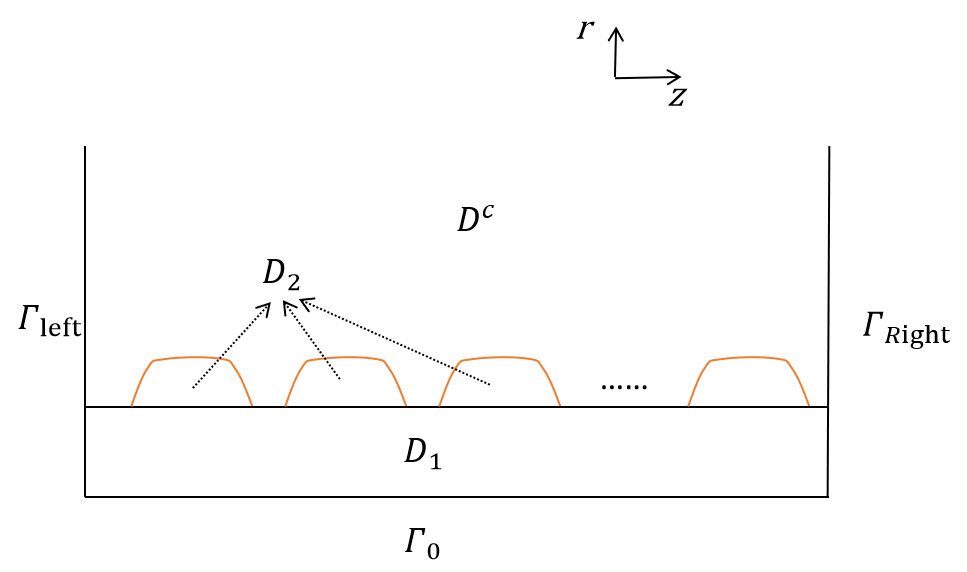} \\
(a) & (b)
\end{tabular}
\caption{Left: a cartoon representation of a neuron with a myelinated axon~\cite{axon}. Right: geometric settings for the mathematical model.}
\label{model}
\end{figure}

Then the main theoretical part of this work lies in proving the well-posedness of the derived electromagnetic problem and the convergence of the solution to an approximate problem resulting from the perfectly-matched-layer (PML) truncation. In both mathematical and engineering communities, related wave propagation problems in electromagnetics, as well as in acoustics and elastodynamics, have been extensively investigated \cite{CK98,monk03,KGBB79}, wherein the Dirichlet-to-Neumann (DtN) map \cite{F83,HW13,KG89} or the PML stretching \cite{BW05,BP07,CZ17} is commonly utilized to truncate the unbounded domain. For the newly derived electromagnetic problem, an exact DtN map and appropriate Sobolev spaces in cylindrical coordinates are introduced to prove the well-posedness of the corresponding variational problem as well as an inf-sup condition related to the weak formulation. The DtN truncated problem can be used for the numerical implementation, however, the DtN map is defined as an infinite series which is nonlocal and requires to be truncated into a finite sum in practical computation. Compared with the DtN technique, The PML method, first proposed by B\'erenger for solving the time-dependent Maxwell equations \cite{B94} and considered in \cite{ZMGT18,ZZML22} for different models of nerve signal propagation, takes advantages in numerical implementation and provides a viable alternative for solving the considered problem. Relying on the error estimate between the exact DtN map and the PML equivalent DtN map, the well-posedness and exponential convergence of the solution to the PML truncated problem are established. Then numerical experiments are reported to both verify our theoretical results and provide efficient simulation results to characterize the signal transmission phenomenon in axon, especially, the existence of myelin sheath can gather the electromagnetic wave to propagate mainly in myelin sheath.

The remainder of this paper is organized as follows. Section~\ref{sec:model} proposes an electromagnetic scattering problem in cylindrical coordinates to model the signal transmission in myelinated axon. Then the well-posedness of the problem in appropriate Sobolev space is investigated in Section~\ref{sec:variational} by utilizing an exact DtN map to reduce the original problem onto a bounded domain. The PML stretching technique is introduced in Section~\ref{sec:trupml} and then the well-posedness of the PML truncated problem, as well as the exponential convergence of the solutions with respect to PML parameters, is proved. Numerical experiments are presented in Section~\ref{sec:numer} to demonstrate the accuracy of the numerical solver and the efficiency of the mathematical model, especially, to provide an intuitive numerical description for the signal transmission in myelinated axon.

\section{Mathematical model}\label{sec:model}

Following the discussion in introduction, as shown in Figure~\ref{model}(b), this section devotes to deriving a new electromagnetic problem in cylindrical coordinates to model the signal propagation in axon.

\subsection{Governing equations}
\label{subsec:govern-eqn}

We begin with the Maxwell's equations given by
\begin{subequations}\label{eq:1st}
\begin{align}
-i\omega\mu \boldsymbol{H} + \curl{\boldsymbol{E}} &= \boldsymbol{0},\label{eq:1st:H}\\
i\omega\varepsilon\boldsymbol{E} + \curl{\boldsymbol{H}} &= \sigma\boldsymbol{E}.\label{eq:1st:E}
\end{align}
\end{subequations}
Here, $\varepsilon$, $\mu$ and $\sigma$ denote the electric permittivity, the magnetic permeability, and the electric conductivity, respectively and we assume that $\mu$ is a positive constant and $\varepsilon,\sigma$ are nonnegative and piece-wise constant. Compared with the axon, the electric conductivity in myelin sheath and water can be ignored and thus, $\sigma=0$ in $D_2$ and $D^c$. By eliminating $\boldsymbol{H}$ or $\boldsymbol{E}$, the second-order equation
\begin{equation}\label{eq:2nd:E}
\curl{\curl{\boldsymbol{E}}} - (k^2 + i\omega\mu\sigma) \boldsymbol{E} = \boldsymbol{0},
\end{equation}
holds for the electric field $\BE$ where $k = \omega\sqrt{\varepsilon\mu}$, and the second-order equation
\begin{equation}\label{eq:2nd:H}
\curl\big[(k^2 + \Vi\omega\mu\sigma)^{-1}\curl\BH\big] - \boldsymbol{H} =  \boldsymbol{0}.
\end{equation}
holds for the magnetic field $\boldsymbol{H}$.

Since axons usually maintain a constant radius, we next consider the Maxwell's equations in cylindrical coordinates and write
\begin{equation}\label{cyl}
\BE=E_r\hat\Br + r^{-1}E_\theta\hat\thetabf +E_z\hat\Bz,\quad
\BH=H_r\hat\Br + r^{-1}H_\theta\hat\thetabf +H_z\hat\Bz,
\end{equation}
where
$\hat\Br$, $\hat\thetabf$, $\hat\Bz$ are the unit vectors in the positive $r$-, $\theta$-, and $z$-directions, respectively. Then for the TM mode, \eqref{eq:2nd:E} is reduced to
\begin{equation}\label{eq:TM-E}
r\frac{\partial}{\partial r}\bigg(\frac{1}{r}\frac{\partial {E_\theta}}{\partial r}\bigg)
+\frac{\partial^2E_\theta}{\partial z^2}
+ (k^2 + i\omega\mu\sigma)E_\theta = 0.
\end{equation}
Similarly, for the TE mode, we can derive from \eqref{eq:2nd:H} that
\begin{equation}\label{eq:TE-H}
r\frac{\partial}{\partial r}\bigg(\frac{1}{(k^2+\Vi\omega\mu\sigma)r}
\frac{\partial H_\theta}{\partial r}\bigg)
+\frac{\partial}{\partial z}\bigg(\frac{1}{k^2+\Vi\omega\mu\sigma}
\frac{\partial H_\theta}{\partial z}\bigg)
+ H_\theta = 0.
\end{equation}
Denote $\Omega = (0, +\infty)\times(0, Z)$. Then in $D^c$ where $\sigma=0$, both (\ref{eq:TM-E}) and (\ref{eq:TE-H}) take a simplified form as follows:
\begin{equation}\label{eq:u}
r\frac{\partial}{\partial r}\bigg(\frac{1}{r}\frac{\partial u}{\partial r}\bigg)
+\frac{\partial^2u}{\partial z^2} + k^2 u = 0, \quad u=E_\theta\;\;\mbox{or}\;\; H_\theta.
\end{equation}

\subsection{Boundary and radiation conditions}\label{subsec:bc}

To complete the modeling of signal propagation in axon, we impose the boundary conditions
\begin{align}\label{eq:z-bc}
\begin{cases}
u(r,0) = u_0(r) & \mbox{on}\quad\Gamma_{\mathrm{left}}, \cr
\frac{\partial u}{\partial z}(r,Z) = u_N(r) & \mbox{on}\quad\Gamma_{\mathrm{right}}^1:= \Gamma_{\mathrm{right}}\cap\partial D_1, \cr
u(r,Z) = u_1(r) & \mbox{on}\quad \Gamma_{\mathrm{right}}^c:=\Gamma_{\mathrm{right}}\cap\partial D^c,
\end{cases}
\end{align}
for $u=E_\theta$ or $u=H_\theta$ on $\Gamma_{\mathrm{left}}$ and $\Gamma_{\mathrm{right}}$, and $u_0(r),u_1(r)$ are compactly supported, i.e., there exists some $R>0$ such that for $r>R$, $u_0(r)=u_1(r)=0$. On the other hand, since both $\BE$ and $\BH$ are bounded at $r=0$, \eqref{cyl} indicates that on $\Gamma_0$,
\begin{equation}\label{EH-bc}
\lim_{r\rightarrow 0} r^{-1}u=0\quad \forall\, z\in [0,Z].
\end{equation}
Moreover, it is necessary to force appropriate radiation condition for $u$ as $r\to\infty$. The commonly used Silver-M\"uller radiation conditions for the electromagnetic fields $\boldsymbol{E},\boldsymbol{H}$ are given by
\begin{equation*}\label{eq:silver-muller}
\lim_{|\boldsymbol{x}|\to\infty}|\boldsymbol{x}|\SN{\curl{\boldsymbol{E}}\times \hat{\boldsymbol{x}}
- \Vi k\boldsymbol{E}} = 0,\quad \lim_{|\boldsymbol{x}|\to\infty}|\boldsymbol{x}|\SN{\curl{\boldsymbol{H}}\times \hat{\boldsymbol{x}} - \Vi k\boldsymbol{H}} = 0.
\end{equation*}
Then we assume that $u=E_\theta$ or $H_\theta$ admits the radiation condition as
\begin{equation}\label{eq:rbc}
\lim_{r\to\infty}\SN{\frac{\partial u}{\partial r} - iku} = 0.
\end{equation}

\section{The well-posedness analysis}
\label{sec:variational}

This section devotes to studying the well-posedness of the electromagnetic model derived in Section~\ref{sec:model} through the variational approach. For simplicity, we only consider the problem of TM mode which consists of the equation \eqref{eq:TM-E}, the boundary conditions (\ref{eq:z-bc})-(\ref{EH-bc}) and the radiation condition (\ref{eq:rbc}). The results for the problem of TE mode can be carried out analogously and thus is omitted here. For the convenience of the following investigations, we denote by $\Omega_r = \{(\xi,\eta): 0<\xi<r,\; 0<\eta <Z\}$ a truncated domain and denote the boundary $\Gamma_r = \{(r,z): 0< z< Z\}, \Sigma_r =\partial\Omega_r\backslash(\Gamma_r\cup\Gamma_0\cup \Gamma_{\mathrm{right}}^1)$. Denote $D_r^c=\Omega_r\backslash\overline{D_1\cup D_2}$. Let $R>0$ be large enough such that $\Supp(u_0)\cup\Supp(u_1)\subset [0,R]$ and $\Supp(\sigma) \subset \Omega_R$.

\subsection{Sobolev spaces on the truncated domain}
\label{subsec:sobolev}

First we introduce some Sobolev spaces and their trace spaces in cylindrical coordinates. We define the space $U_R$ equipped with a weighted $L^2$-norm as
\begin{align*}
U_R = \big\{v\in\Ltwo[\Omega_R]:\N{v}_{U_R}<\infty\big\},\quad \N{v}^2_{U_R} = \int_{\Omega_R} \SN{u}^2r^{-1}\D r\D z .
\end{align*}
The subspaces $V_R$ and $V_{R,0}$ equipped with a weighted  $H^1$-norm are defined as
\begin{align*}
V_R= \big\{v\in\Hone[\Omega_R]:\N{v}_{V_R}<\infty,\; \lim_{r\rightarrow 0} r^{-1}v=0\big\},
\quad
V_{R,0}=\big\{v\in V_R: v|_{\Sigma_R}=0\big\},
\end{align*}
and
\ben
\N{v}_{V_R}^2 = \int_{\Omega_R}\bigg(\Big|\frac{\partial u}{\partial r}\Big|^2
+\Big|\frac{\partial u}{\partial z}\Big|^2\bigg) r^{-1}\D r\D z + k^2\N{v}^2_{U_R} .
\een

\begin{lemma}\label{lem:compact}
The injection from $V_R$ into $U_R$ is compact.
\end{lemma}
\begin{proof}
Let $\{v_n\}_{n=1}^\infty\subset V_R$ be a bounded sequence. The compact injection $V_R\hookrightarrow\hookrightarrow\Ltwo[\Omega_R]$ implies that there exists a subsequence $\{v_{n_k}\}$ which is convergent in $\Ltwo[\Omega_R]$ and thus, $\NLtwo[\Omega_R]{e_{ij}}\to 0$ as $i,j\to \infty$ where $e_{ij}:=v_{n_i}-v_{n_j}$. Since $e_{ij}$ vanishes on $\Gamma_0$ for all $i$ and $j$, we deduce that
\begin{align*}
\int_{\Omega_R}\frac{1}{r}\SN{e_{ij}(r,z)}^2\D r\D z
=\,&\int_{\Omega_R}\frac{1}{r}\bar{e}_{ij}(r,z)
\bigg[\int_0^r\frac{\partial e_{ij}}{\partial t}(t,z)\D t\bigg]\D r\D z \\
\le\,& \NLtwo[\Omega_R]{e_{ij}}
\bigg[\int_{\Omega_R} \frac{1}{r^2}\Big|\int_0^r
\frac{\partial e_{ij}}{\partial t}(t,z)\D t\Big|^2
\D r\D z\bigg]^{1/2} \\
\le\,&   \NLtwo[\Omega_R]{e_{ij}}  \bigg[\int_{\Omega_R} \frac{1}{r}\int_0^r \Big|
\frac{\partial e_{ij}}{\partial t}(t,z)\Big|^2\D t
\D r\D z\bigg]^{1/2} \\
\le\,&   \NLtwo[\Omega_R]{e_{ij}}  \bigg[\int_{\Omega_R} \int_0^R \frac{1}{t}\Big|
\frac{\partial e_{ij}}{\partial t}(t,z)\Big|^2\D t
\D r\D z\bigg]^{1/2} \\
=\,& R^{1/2}\NLtwo[\Omega_R]{e_{ij}}\N{e_{ij}}_{V_R}.
\end{align*}
Hence, $\{v_{n_k}\}$ is a Cauchy sequence under the norm $\N{\cdot}_{U_R}$. It is easy to see that $U_R$ is a Banach space. Therefore, the subsequence $\{v_{n_k}\}$ is convergent in $U_R$ which means that $V_R$ is compactly embedded into $U_R$.
\end{proof}

Now we study the Sobolev spaces on $\Gamma_R$.
For any $v\in\Ltwo[\Gamma_R]$, we can write it into a Fourier series
$v=\sum_{m=0}^\infty \big[v_{1,m}\cos(m\pi z/Z) + v_{2,m}\sin(m\pi z/Z)\big]$.
An equivalent norm on $\Ltwo[\Gamma_R]$ is defined by the Parseval identity
\ben
\N{v}_{L^2(\Gamma_R)}^2= \frac{Z}{2}\sum_{m=0}^\infty\big(\SN{v_{1,m}}^2 +\SN{v_{2,m}}^2\big).
\een
Similarly, the norm on $H^{1/2}(\Gamma_R)$ is defined as
\ben
\N{v}_{H^{1/2}(\Gamma_R)}^2= \sum_{m=0}^\infty m\big(\SN{v_{1,m}}^2 +\SN{v_{2,m}}^2\big).
\een
Let $H^{1/2}_0(\Gamma_R)$ denote the closure of $C^\infty_0(\Gamma_R)$ in $H^{1/2}(\Gamma_R)$, namely,
\begin{equation}\label{norm-half}
H^{1/2}_0(\Gamma_R) := \ol{C^\infty_0(\Gamma_R)}^{\N{\cdot}_{H^{1/2}(\Gamma_R)}}.
\end{equation}
It is easy to see that $H^{1/2}_0(\Gamma_R)$ is the trace space of $V_{R,0}$ on $\Gamma_R$.
The dual space of $H^{1/2}_0(\Gamma_R)$ is denoted by $H^{-1/2}(\Gamma_R):=\big[H_0^{1/2}(\Gamma_R)]'$. Any $w\in H^{1/2}_0(\Gamma_R)$ admits a Fourier expansion $w=\sum_{m=1}^\infty w_m\sin(m\pi z/Z)$. The equivalent norms on $H^{1/2}_0(\Gamma_R)$ and $H^{-1/2}(\Gamma_R)$ are, respectively, defined as
\ben
\N{w}_{H^{1/2}_0(\Gamma_R)}^2= \sum_{m=1}^\infty m\SN{w_m}^2,\quad
\N{w}_{H^{-1/2}(\Gamma_R)}^2 = \sum_{m=1}^\infty m^{-1}\SN{w_m}^2.
\een

\subsection{Dirichlet-to-Neumann (DtN) map and truncated problem}
\label{subsec:dtn}

Since $u(r,0)=u(r,Z) =0$ for $r\ge R$, we can write $u$ into a Fourier series $u(r,z)=\sum\limits_{m=1}^{\infty}u_m(r)\sin(m\pi z/Z)$ for $r\ge R, 0\le z\le Z$. Substituting the series into
\eqref{eq:u}, we get an ordinary differential equation
\begin{equation}\label{eq:am}
u_m''(r)-r^{-1}u_m'(r) + k^2_m u_m(r) = 0,
\quad k_m =
\begin{cases}
\sqrt{k^2-m^2\pi^2/Z^2} & \hbox{if}\;\; k> m\pi/Z,
    \vspace{1mm} \\
\Vi\sqrt{m^2\pi^2/Z^2-k^2} & \hbox{otherwise}.
\end{cases}
\end{equation}
To eliminate resonance mode, we assume that $k\ne m\pi/Z$ for all $m\ge 1$.
Using \eqref{eq:rbc}, the outgoing solution of (\ref{eq:am}) is given by
\begin{equation}\label{eq:Hm}
u_m(r)= a_mrH^{(1)}_1(k_mr),
\end{equation}
where $H^{(1)}_1$ is the first order Hankel function of the first kind.
It is easy to see that
\begin{equation}\label{dru}
\frac{\partial}{\partial r}(r^{-1}u) = r^{-2}
\sum_{m=1}^{\infty} h(k_mr) u_m(r)\sin(m\pi z/Z),
\quad
h(t):= t H^{(1)\,\prime}_1(t)/H^{(1)}_1(t).
\end{equation}

\begin{remark}
Since $\big|H^{(1)}_1(t)\big|\sim t^{-1/2}$ as $t\to+\infty$, equation \eqref{eq:Hm} implies that $\SN{u_m(r)}$ grows at the rate $r^{1/2}$ as $r\to +\infty$. In view of \eqref{cyl}, it is $r^{-1}u$ that represents the angular component of $\BE$ or $\BH$. Therefore, the growing rate of $u_m(r)$ is reasonable.
\end{remark}

For any $\eta\in H_0^{1/2}(\Gamma_R)$ which admits a Fourier series $\eta=\sum_{m=1}^{\infty} \eta_m \sin(m\pi z/Z)$, we define a Dirichlet-to-Neumann (DtN) map ${T}:H^{1/2}_0(\Gamma_R)\to H^{-1/2}(\Gamma_R)$ by
\ben
T\eta=\frac{\partial}{\partial r}(r^{-1}v)\big|_{\Gamma_R},
\een
where $v$ is the solution of \eqref{eq:u} in $\Omega_R^c=\Omega^c\backslash\overline{\Omega_R}$ satisfying the Dirichlet boundary conditions $v(R,z)=\eta$ for $0\le z\le Z$ and $v(r,0)=v(r,Z)=0$ for $r\ge R$, and the radiation condition (\ref{eq:rbc}). Using (\ref{dru}), the DtN map can be expressed as
\begin{equation}\label{DtN}
{T}\eta =\sum_{m=1}^{\infty} h(k_mR) \eta_m \sin(m\pi z/Z).
\end{equation}
Utilizing the continuity of the fields crossing the boundary $\Gamma_R$, we can reformulate the model of TM mode on the truncated domain
\begin{subequations}\label{eq:modelR}
\begin{align}
&r\frac{\partial}{\partial r}\bigg(\frac{1}{r}\frac{\partial u}{\partial r}\bigg)
+\frac{\partial^2u}{\partial z^2}
+ (k^2 + i\omega\mu\sigma) u = 0\quad\hbox{in}\;\;\Omega_R,\label{eq:eqn-u} \\
& \lim_{r\rightarrow 0} r^{-1}u=0 \quad\hbox{on}\;\;\Gamma_0,\label{eq:bc-u-hR1} \\
& u=u_D \quad\hbox{on}\;\;\Sigma_R,\label{eq:bc-u-hR2} \\
& \frac{\partial u}{\partial z}=u_N\quad\hbox{on}\;\;\Gamma_{\mathrm{right}}^1,\label{eq:bc-u-Neu} \\
& \frac{\partial}{\partial r}(r^{-1}u) = R^{-2}{T}u
    \quad\hbox{on}\;\;\Gamma_R,\label{eq:bc-u-TR}
\end{align}
\end{subequations}
where $u_D\in H^{1/2}(\Sigma_R)$ is defined as
\ben
u_D=u_0\;\; \hbox{on} \;\; \Gamma_{\mathrm{left}},\quad
u_D=u_1\;\; \hbox{on} \;\; \Gamma_{\mathrm{right}}^c.
\een

Multiplying both sides of \eqref{eq:eqn-u} with $v/r$ where $v\in V_{R,0}$, integrating the result on $\Omega_R$ and taking integration by part, we can obtain the corresponding variational problem of \eqref{eq:modelR} as follows: find $u\in V_R$ such that $u=u_D$ on $\Sigma_R$ and
\begin{align}\label{weak}
a(u,v) =\int_{\Gamma_{\mathrm{right}}^1} u_N\bar{v}\D r \quad \forall\, v\in  V_{R,0},
\end{align}
where $a$ is a sesquilinear form on $V_R\times V_R$ defined by
\begin{align}\label{eq:a}
a(u,v) = \int_{\Omega_R}\bigg[\frac{\partial u}{\partial r}
\frac{\partial\bar{v}}{\partial r}
+\frac{\partial u}{\partial z}\frac{\partial \bar{v}}{\partial z}
- (k^2 + \Vi\omega\mu\sigma) u \bar{v}\bigg]\frac{1}{r}\, \D r\D z
-\frac{1}{R^2}\langle Tu + u, v\rangle_{\Gamma_R}.
\end{align}
Here $\displaystyle \langle \xi, v\rangle_{\Gamma_R}:=\int_{\Gamma_R}\xi\bar v \D z$ stands for the duality product if $\xi\in H^{-1/2}(\Gamma_R)$ and $v\in H_0^{1/2}(\Gamma_R)$ or the inner product if $\xi,v\in L^2(\Gamma_R)$.

\subsection{The well-posedness of problem (\ref{weak})}
\label{subsec:infsup}

In this subsection, we shall prove the inf-sup condition for the sesquilinear form $a$ and establish the well-posedness of problem \eqref{weak}. For any $v,w\in H^{1/2}_0(\Gamma_R)$, we always write $v = \sum_{m=1}^{\infty}v_m \sin(m\pi z/Z)$ and $w = \sum_{m=1}^{\infty}w_m \sin(m\pi z/Z)$.
First we prove some useful results of the DtN map.

\begin{lemma}\label{lem:ImT}
$\Im \langle{T}v, v\rangle_{\Gamma_R}\ge  0$ for any $v\in H_0^{1/2}(\Gamma_R)$ and
$\Im \langle{T}v, v\rangle_{\Gamma_R} =  0$ implies $v\in \Span{\sin(m\pi z/Z): m > kZ/\pi}$.
\end{lemma}
\begin{proof}
It follows from \eqref{DtN} that $Im \langle{T}v, v\rangle_{\Gamma_R}
= \frac{Z}{2}\sum_{m=1}^{\infty} \Im [h(k_mR)]\SN{v_m}^2$ for any $v\in H^{1/2}_0(\Gamma_R)$. For $k> m\pi/Z$, we have $k_m>0$. Recalling that for $t>0, l\in\mathbb{Z}$, $H^{(1)}_{l}(t) =J_{l}(t) + \Vi Y_{l}(t)$ and
$J_{l}(t)Y_{l-1}(t)-J_{l-1}(t)Y_{l}(t)=2/(\pi t)$, we have
\begin{align*}
H^{(1)}_{l-1}(t)\ol{H^{(1)}_{l}(t)}
= J_{l-1}(t)J_{l}(t) + Y_{l-1}(t)Y_{l}(t) +2\Vi /(\pi t)\quad
\forall\,t>0,  l\in\mathbb{Z}.
\end{align*}
Therefore, from the relation $t H^{(1)\,\prime}_1(t)=tH^{(1)}_0(t) - H^{(1)}_1(t)$, we know
\begin{equation}\label{Im-hm}
\Im[h(k_mR)] =\SN{H^{(1)}_1(t)}^{-2}
\Im\Big[t H^{(1)}_0(t)\ol{H^{(1)}_1(t)}\Big]=\frac{2}{\pi}\SN{H^{(1)}_1(k_mR)}^{-2}>0.
\end{equation}
For $k<m\pi/Z$, we recall the modified Bessel functions $K_l$ which satisfy
\begin{equation}\label{Km}
H^{(1)}_l(\Vi t) = \frac{2}{\pi}e^{-(l+1)\pi/2 \Vi} K_l(t),\quad t>0.
\end{equation}
Since $K_l(t)$ are real for all $l$ and $t>0$, from \eqref{dru} and \eqref{Km}, we easily know that  $\Im[h(k_mR)]=0$.
The proof is complete.
\end{proof}

\begin{lemma}\label{lem:Ct-cont}
There exists a constant $C>0$ depending only on $k$, $R$, and $Z$ such that
\ben
\N{Tv}_{H^{-1/2}(\Gamma_R)}\le C\N{v}_{H^{1/2}(\Gamma_R)}
\quad \forall\,v\in H^{1/2}_0(\Gamma_R).
\een
\end{lemma}
\begin{proof}
Note that for $1\le m<kZ/\pi$, $k_m>0$. We know from \cite[(3.17)]{mel10} that
\begin{align}\label{Re-hm}
\frac{4k_m^2R^2}{4k_m^2R^2+3}
\le -\Re h(k_mR) \le \frac{1}{2} +\frac{9}{16k_m^2R^2}.
\end{align}
Then \eqref{Im-hm} and \eqref{Re-hm} yield that there exists a constant $C>0$ depending only on $k$ such that$\SN{h(k_mR)}\le C$. For $m>kZ/\pi$, from \eqref{Km}, the positivity of $K_m(t)$ for all $m\ge 0,t>0$ and \cite[eq. (10.37.1)]{olv10}, we have
\ben
-h(k_mR) =1+ |k_m| R \frac{K_0(|k_m|R)}{K_1(|k_m|R)}
\le 1+ |k_m| R\le 1+ m\pi R/Z.
\een
For any $v,w\in H^{1/2}_0(\Gamma_R)$, we have
\ben
\SN{\langle Tv,w\rangle_{\Gamma_R}} = \frac{Z}{2}\sum_{m=1}^{\infty} \SN{h(k_mR)v_m\ol{w_m}}
&&\le C \sum_{m=1}^{\infty}m\SN{v_m\ol{w_m}}\\
&&\le C\N{v}_{H^{1/2}(\Gamma_R)}\N{w}_{H^{1/2}(\Gamma_R)}.
\een
This leads to
\ben
\N{Tv}_{H^{-1/2}(\Gamma_R)} = \sup_{w\in H^{1/2}_0(\Gamma_R)}
\frac{\SN{\langle Tv,w\rangle_{\Gamma_R}}}{\N{w}_{H^{1/2}(\Gamma_R)}}
\le C\N{v}_{H^{1/2}(\Gamma_R)}.
\een
The proof is finished.
\end{proof}

The proof of Lemma~\ref{lem:Ct-cont} also indicates the following useful result.
\begin{lemma}\label{lem:positive}
It holds that $-\Re \langle{T}v, v\rangle_{\Gamma_R}\ge 0$ for all $v\in H^{1/2}_0(\Gamma_R)$.
\end{lemma}

The uniqueness of the variational problem \eqref{weak} is given in the following lemma.

\begin{lemma}\label{lem:unique}
The variational problem \eqref{weak} has at most one solution.
\end{lemma}
\begin{proof}
Since \eqref{weak} is a linear problem, it is suffices to show that $u_D=u_N=0$ implies $u\equiv 0$.
Now we suppose $u\in V_{R,0}$ and take $v=u$ in \eqref{weak}. The imaginary part of the equation shows
\begin{align*}
\omega\mu\int_{\Omega_R} \sigma \SN{u}^2  r^{-1}\D r\D z
+ R^{-2}\Im\langle Tu, u\rangle_{\Gamma_R} =0 .
\end{align*}
From Lemma~\ref{lem:ImT}, we infer that $u\equiv 0$ in $D_1$.
From \eqref{eq:u}, we have
\begin{equation*}
r\frac{\partial}{\partial r}\bigg(\frac{1}{r}\frac{\partial u}{\partial r}\bigg)
+\frac{\partial^2u}{\partial z^2} + k^2 u = 0\quad
\hbox{in}\;\;\Omega_R\backslash\ol{D_1}.
\end{equation*}
Since $u\equiv 0$ in $D_1$, we also have
\begin{equation}\label{eq-OmegaR}
r\frac{\partial}{\partial r}\bigg(\frac{1}{r}\frac{\partial u}{\partial r}\bigg)
+\frac{\partial^2u}{\partial z^2}+ k^2 u = 0\quad
\hbox{in}\;\;\Omega_R.
\end{equation}

Take a $\Bp\in \partial D_{1}$ and an open disk $B_\delta(\Bp)$ with the radius $\delta$ and the the center being $\Bp$. Assume $B_\delta(\Bp)\subset \Omega_R$ without loss of generality. There exists a constant depending on $\delta$ and $k$ such that
\ben
\bigg|\frac{\partial^2u}{\partial r^2}+\frac{\partial^2u}{\partial z^2}\bigg|
\le C\bigg(\bigg|\frac{\partial u}{\partial r}\bigg| + |u|\bigg) \quad
\hbox{a.e. in}\;B_\delta(\Bp).
\een
Note that $u\equiv 0$ in $B_\delta(\Bp)\cap D_1$.
By the unique continuation theory (see Lemma~4.15 in \cite{monk03}, page 93), we have $u\equiv 0$ in
$B_\delta(\Bp)$. Moreover, we can extend the arguments from $B_\delta(\Bp)$ to $\Omega_R$
and end up with $u\equiv 0$ in $\Omega_R$.
\end{proof}

Now, we are ready to show the inf-sup condition for the sesquilinear form $a$ and establish the well-posedness of problem \eqref{weak}.

\begin{theorem}\label{thm:infsup}
There exists a unique solution to the variational problem \eqref{weak}. Moreover, the inf-sup condition
\begin{align*}
\sup_{0\ne v\in V_{R,0}}\frac{\SN{a(w,v)}}{\N{v}_{V_R}} \ge C_{\inf}\N{w}_{V_R},
\quad \forall\,w\in V_{R,0},
\end{align*}
holds where $C_{\inf}>0$ is a constant depending only on $k$, $R$, $Z$ and material parameters.
\end{theorem}
\begin{proof}
It suffices to prove that, for any $\ell\in V_{R,0}'$, there exists a unique solution $w\in V_{R,0}$ to the problem
\begin{align}\label{eq:ell}
a(w,v) =\ell(v)\quad \forall\,v\in V_{R,0},
\end{align}
and that the solution satisfies $\N{w}_{V_R}\le C\N{\ell}_{V_{R,0}'}$. To do this, we define another sesquilinear form $a_+$ on $V_R\times V_R$ as
\begin{align}\label{eq:ap}
a_+(u,v) = \int_{\Omega_R}\bigg[\frac{\partial u}{\partial r}
\frac{\partial\bar{v}}{\partial r}
+\frac{\partial u}{\partial z}\frac{\partial \bar{v}}{\partial z} +
(k^2 - \Vi\omega\mu\sigma) u \bar{v}\bigg]\frac{1}{r}\D r\D z
-\frac{1}{R^2}\langle Tu, v\rangle_{\Gamma_R}.
\end{align}
By Lemmas~\ref{lem:ImT}-\ref{lem:positive}, there is a generic constant $C>0$ depending only on $k$, $R$, $Z$ and material parameters such that
\begin{align}\label{ap}
|a_+(v,v)|\ge \Re a_+(v,v) \ge \N{v}_{V_R}^2,\quad
\SN{a_+(u,v)} \le C\N{u}_{V_R}\N{v}_{V_R}.
\end{align}

Let $\Ck_0: U_R\to V_{R,0}$ and $\Ck_1: \Ltwo[\Gamma_R]\to V_{R,0}$ be two operators defined as follows: for $w\in U_R$ and $\eta\in \Ltwo[\Gamma_R]$, $\Ck_0w$ and $\Ck_1\eta$ are the unique solutions to the following two problems, respectively,
\begin{align}\label{K0}
a_+(\Ck_0w,v) =2k^2\int_{\Omega_R}w\bar{v}\frac{1}{r}\D r\D z,\quad
a_+(\Ck_1\eta,v) =R^{-2}\langle\eta, v\rangle_{\Gamma_R} \quad
\forall\,v\in V_{R,0}.
\end{align}
Then \eqref{ap} together with Lax-Milgram theorem implies that $\Ck_0$ and $\Ck_1$ are continuous operators, namely,
\ben
\N{\Ck_0w}_{V_R}\le C\N{w}_{U_R}, \quad
\N{\Ck_1\eta}_{V_R}\le C\NLtwo[\Gamma_R]{\eta}.
\een
By Lemma~\ref{lem:compact} and the compact injection $H^{1/2}(\Gamma_R)\hookrightarrow\hookrightarrow\Ltwo[\Gamma_R]$, both $\Ck_0: V_{R,0}\to V_{R,0}$ and $\Ck_1: H_0^{1/2}(\Gamma_R)\to V_{R,0}$ are compact operators.

Let $\gamma_R$ denote the trace operator which maps $V_{R,0}$ onto $H^{1/2}_0(\Gamma_R)$ continuously. We can write problem \eqref{eq:ell} into an equivalent operator equation
\begin{align}\label{eq:opr}
w -(\Ck_0+\Ck_1\gamma_R)w =f,
\end{align}
where $f\in V_{R,0}$ is the unique solution to the problem
\begin{align*}
a_+(f,v) =\ell(v)\quad \forall\,v\in V_{R,0}.
\end{align*}
Since $\Ck_0+\Ck_1\gamma_R: V_{R,0}\to V_{R,0}$ is a compact operator, \eqref{eq:opr} is a Fredholm equation of second kind. By Lemma~\ref{lem:unique} and Fredholm alternative theorem, we conclude that there exists a unique solution to the variational problem (\ref{eq:ell}). Let $\Ci$ denote the identity operator on $V_{R,0}$. The arbitrariness of $\ell\in V_{R,0}'$ shows that $(\Ci-\Ck_0-\Ck_1\gamma_R)^{-1}$ exists and is a continuous operator from $V_{R,0}$ to $V_{R,0}$. We end up with
\ben
\N{w}_{V_R}\le C\N{f}_{V_R}\le C\N{\ell}_{V_{R,0}'}
=C\sup_{0\ne v\in V_{R,0}}\frac{|a(w,v)|}{\N{v}_{V_R}}.
\een
The proof is complete.
\end{proof}

\begin{remark}
In this section, we establish the well-posedness for the derived electromagnetic model by introducing the exact DtN map to truncate the unbounded domain. Incorporating with the finite element method (FEM), which is called DtN-FEM, the variational problem \eqref{weak} can be used for the numerical simulation. Noting that the discretization of the sesquilinear form $\langle Tu,v\rangle_{\Gamma_R}$ is a nonlocal integral, in the rest of this work we utilize an alternative way, the PML method, to truncate the unbounded domain $\Omega^c$ and study the convergence of this method. The analysis and application of DtN-FEM, including the truncation of the infinite series in the DtN map and adaptivity, are left for future works.
\end{remark}

\section{The truncated PML problem and convergence study}
\label{sec:trupml}

Besides the DtN map discussed in Section~\ref{sec:variational}, this section proposes an approximate problem applying the PML truncation strategy. Suppose $\rho\ge 2R$ and let $\Omega_\rho:=\{0<r<\rho, 0<z<Z\}$ be the domain in which the truncated PML problem is formulated. Denote $\Omega_{\pml}:= \Omega_\rho\backslash\bar\Omega_R$ the PML region with
$d=\rho-R$ being the thickness of $\Omega_{\pml}$. To derive the truncated PML problem, we introduce the following complex stretching of radial coordinate. Define
\begin{align}\label{eq:tr}
\tilde{r}= F(r):= r + (1+\Vi)\chi(r),\quad
\chi(r) =
\begin{cases}
0 & \hbox{if}\;\; r\le R,   \\
\chi_0(r-R)^2 & \hbox{if}\;\; r> R,
\end{cases}
\end{align}
where $\chi_0> 0$ is a constant. Clearly, $\chi$ is $C^1$-smooth in $(0,+\infty)$. Denote
\begin{align}\label{Jac}
\alpha(r) = F'(r) = 1 + (1+\Vi)\chi'(r),\quad
\beta(r) = \tilde{r}/r =  1 + (1+\Vi)\chi(r)/r.
\end{align}
We have already assumed that $k\ne m\pi/Z$ for all $m\ge 1$. Without loss of generality, we additionally assume
\begin{align}\label{kappa}
\kappa\chi_0R\ge 1, \quad \kappa:=\min\limits_{m\ge 1}\SN{k_m}.
\end{align}
%

\subsection{The approximate problem}

Note from \eqref{eq:Hm} that, for $r\ge R$, the exact solution to problem \eqref{eq:modelR} can be represented as
\begin{align}\label{exp-u}
u(r,z) =r \sum_{m=1}^\infty a_m H^{(1)}_1(k_mr)\sin(m\pi z/Z) \quad \forall\, r\ge R.
\end{align}
With this explicit representation, we can define the analytic continuation of $u$ from the real variable $r$ to the complex variable $\tilde{r}$ by
\begin{align}\label{exp-ut}
u(\tilde r,z) =\tilde r \sum_{m=1}^\infty a_m H^{(1)}_1(k_m\tilde{r})\sin(m\pi z/Z) \quad \forall\, r\ge R.
\end{align}
Then by the chain rule, it is easy to see that $\tilde u(r,z):= u(\tilde r, z)$ satisfies the modified equation
\begin{align}\label{tu-eqn}
\frac{r\beta}{\alpha}\frac{\partial}{\partial r}
\bigg(\frac{1}{r\alpha\beta}\frac{\partial\tilde{u}}{\partial r}\bigg)
+\frac{\partial^2\tilde{u}}{\partial z^2}
 + (k^2 + i\omega\mu\sigma) \tilde{u} = 0 \quad \hbox{in}\;\;\Omega .
\end{align}
To truncate the unbounded domain $\Omega$, it is reasonable to impose the Dirichlet boundary condition $\hat{u} = 0$ on $\Gamma_\rho$ regarding to the exponential decay of the Hankel functions with a complex argument. Hence, the approximate problem to \eqref{eq:modelR} is proposed as follows
\begin{subequations}\label{pro-pml}
\begin{align}
&\frac{r\beta}{\alpha}\frac{\partial}{\partial r}
\bigg(\frac{1}{r\alpha\beta}\frac{\partial\hat{u}}{\partial r}\bigg)
+\frac{\partial^2\hat{u}}{\partial z^2}
+ (k^2 + i\omega\mu\sigma)\hat{u} = 0 \quad \hbox{in}\;\;\Omega_\rho,\label{eqn-pml} \\
& \lim_{r\rightarrow 0} r^{-1}\hat u=0 \quad\hbox{on}\;\;\Gamma_0, \label{eq:bc-pml1}\\
& \hat{u}=u_D \quad\hbox{on}\;\;\Sigma_\rho,\quad
 \hat{u} = 0    \quad\hbox{on}\;\;\Gamma_\rho,\label{eq:bc-pml2}\\
& \frac{\partial \hat{u}}{\partial z}=u_N\quad\hbox{on}\;\;\Gamma_{\mathrm{right}}^1.\label{eq:bc-pml-Neu}
\end{align}
\end{subequations}
Letting the spaces $V_\rho,V_{\rho,0}$, as well as the equipped norms, be defined analogous to $V_R,V_{R,0}$, respectively, by replacing the domain $\Omega_R$ by $\Omega_\rho$ and denote $V_\rho^0:=\{v\in V_{\rho,0}: v|_{\Gamma_\rho}=0\}$. Then the weak formulation of the problem (\ref{pro-pml}) reads: find $\hat u\in V_\rho$ such that $\hat{u}=u_D$ on $\Sigma_\rho$ and
\begin{align}\label{weak:pro-pml}
a_{\rho}(\hat{u},v) =\int_{\Gamma_{\mathrm{right}}^1} u_N\bar{v}\D r \quad
\forall\,v\in V_\rho^0,
\end{align}
where
\begin{align*}
a_{\rho}(\hat{u},v) := \int_{\Omega_\rho}\frac{1}{r\beta}\bigg(
\frac{1}{\alpha}\frac{\partial \hat{u}}{\partial r}\frac{\partial\bar{v}}{\partial r}
+\alpha \frac{\partial\hat{u}}{\partial z} \frac{\partial\bar{v}}{\partial z}
-\alpha (k^2 + i\omega\mu\sigma)\hat{u}\bar{v}\bigg)\D r\D z.
\end{align*}

The purpose of the remaining parts of this section is to study the well-posedness of the approximate problem \eqref{pro-pml} and its convergence, i.e., the error estimate between the exact solution $u$ of problem (\ref{eq:modelR}) and the approximate solution $\hat{u}$ of problem \eqref{pro-pml}. Since problem (\ref{eq:modelR}) is defined on $\Omega_R$, we next reformulate \eqref{pro-pml} into a problem on $\Omega_R$ utilizing the DtN map strategy. Let $\hat{T}:H^{1/2}_0(\Gamma_R)\to H^{-1/2}(\Gamma_R)$ be a DtN map defined as, for any $\eta\in H^{1/2}_0(\Gamma_R)$,
\begin{align}\label{hT}
\hat{T}\eta := R^2\frac{\partial}{\partial r}(r^{-1}w)\big|_{\Gamma_R},
\end{align}
where $w$ is the solution to the Dirichlet problem in the PML
\begin{subequations}\label{pro-hT}
\begin{align}
&\frac{r\beta}{\alpha}\frac{\partial}{\partial r}
\bigg(\frac{1}{r\alpha\beta}\frac{\partial w}{\partial r}\bigg)
+\frac{\partial^2w}{\partial z^2} + k^2 w = 0
\quad \hbox{in}\;\;\Omega_{\pml}, \label{eqn-hT} \\
&w=\eta \quad \hbox{on}\;\;\Gamma_R, \label{bc0-hT}\\
&w=0 \quad \hbox{on}\;\;\partial\Omega_\pml\backslash\Gamma_R.\label{bc1-hT}
\end{align}
\end{subequations}
The well-posedness of problem \eqref{pro-hT}, which will be addressed in the next subsection, ensures that the DtN operator $\hat{T}$ is well-defined. Then the continuity of the fields crossing the boundary $\Gamma_R$ indicates that we can reformulate the problem \eqref{pro-pml} in $\Omega_R$ as follows:
\begin{subequations}\label{model-hR}
\begin{align}
&r\frac{\partial}{\partial r}
\bigg(\frac{1}{r}\frac{\partial\hat{u}}{\partial r}\bigg)
+\frac{\partial^2\hat{u}}{\partial z^2}
+ (k^2 + i\omega\mu\sigma)\hat{u} = 0 \quad \hbox{in}\;\;\Omega_R,\label{eqn-hR} \\
& \lim_{r\rightarrow 0} r^{-1}\hat u=0 \quad\hbox{on}\;\;\Gamma_0, \label{eq:bc0-hR1}\\
& \hat{u}=u_D \quad\hbox{on}\;\;\Sigma_R, \label{eq:bc0-hR2} \\
& \frac{\partial \hat{u}}{\partial z}=u_N\quad\hbox{on}\;\;\Gamma_{\mathrm{right}}^1,\label{eq:bc0-Neu} \\
& \frac{\partial}{\partial r}\big(r^{-1} \hat{u}\big) =R^{-2}\hat{T}\hat{u}    \quad\hbox{on}\;\;\Gamma_R.\label{eq:bc1-hR}
\end{align}
\end{subequations}
This leads us to study the well-posedness of problem \eqref{model-hR} and the error estimate between the exact solution $u$ of problem (\ref{eq:modelR}) and the approximate solution $\hat{u}$ of problem \eqref{model-hR}.

\subsection{The well-posedness of problem (\ref{pro-hT})}

To establish the well-posedness of problem \eqref{pro-hT}, we use separation of variables and the Fourier expansions of $w$ and $\eta$ formulated as
\begin{align}\label{exp-w}
w(r,z)=\sum_{m=1}^\infty w_m(r)\sin(m\pi z/Z),\quad \eta(z) =\sum_{m=1}^\infty \eta_m\sin(m\pi z/Z),
\end{align}
to reduce the problem \eqref{pro-hT} into a system of ordinary differential equations, for $m\ge 1$,
\begin{subequations}\label{pro-wm}
\begin{align}
&\frac{r\beta}{\alpha}\frac{\partial}{\partial r}
\bigg(\frac{1}{r\alpha\beta}\frac{\partial w_m}{\partial r}\bigg) + k_m^2 w_m = 0
\quad \hbox{for}\;\;R<r<\rho, \label{eqn-wm} \\
&w_m(R)=\eta_m, \quad w_m(\rho)=0.\label{bc-wm}
\end{align}
\end{subequations}
The variational formulation of \eqref{pro-wm} is given as follows: find $w_m\in H^1((R,\rho))$ which satisfies \eqref{bc-wm} and
\begin{align}\label{weak-wm}
\mathscr{A}_m(w_m,v):=\int_R^\rho \frac{1}{r\beta}\big(\alpha^{-1} w_m' \ol{v'}
-\alpha k_m^2 w_m\bar{v}\big) \D r= 0
\quad \forall\,v\in\zbHone[(R,\rho)].
\end{align}
We first prove the well-posedness of the problem \eqref{pro-wm} for each $m\ge 1$.

\begin{lemma}\label{lem:avv}
There exist two positive constants $C_0,C_1$ independent of $m$, $\chi_0$, and $\rho$ such that, for any $v\in\zbHone[(R,\rho)]$,
\begin{align}
&\Re \mathscr{A}_m(v,v)\ge \int_R^\rho \xi\SN{v'}^2\D r - C_0k_m^2\int_R^\rho\eta\SN{v}^2\D r
\quad \hbox{if}\;\; m > kZ/\pi, \label{ieq:avv-1}\\
&\Re \mathscr{A}_m(v,v)-C_1 d^3\SN{\alpha(\rho)}^{2}\Im \mathscr{A}_m(v,v) \nonumber\\
&\ge \frac{1}{2}\int_R^\rho \xi\SN{v'}^2\D r
    +C_0 d^3\SN{\alpha(\rho)}^{2}\int_R^\rho \eta\chi'\SN{v}^2\D r
    \quad \hbox{if}\;\; m < kZ/\pi, \label{ieq:avv-2}
\end{align}
where $\xi =(r+r\chi'+ \chi)/\SN{r\alpha\beta}^2$ and
$\eta =(r+r\chi'+ \chi+ 2\chi\chi')/\SN{r\beta}^2$. As a result, there exists a unique solution to the problem \eqref{pro-wm} for all $m\ge 1$.
\end{lemma}
\begin{proof}
Write $\xi_1= (r\chi'+ \chi+ 2\chi\chi')/\SN{r\alpha\beta}^2$ and
$\eta_1= (r\chi'- \chi)/\SN{r\beta}^2$ for convenience.
It is easy to see
\begin{align}
& \Re \mathscr{A}_m(v,v)= \int_R^\rho
\Big(\xi\SN{v'}^2 -k_m^2\eta\SN{v}^2\Big),\label{ReIm-am1}\\
& \Im \mathscr{A}_m(v,v)= -\int_R^\rho \Big(\xi_1\SN{v'}^2 +k_m^2\eta_1\SN{v}^2\Big)\label{ReIm-am2}.
\end{align}

Since $k_m^2= k^2-m^2\pi^2/Z^2<0$ for $m > kZ/\pi$, \eqref{ieq:avv-1} comes directly from \eqref{ReIm-am1}. Next, we prove \eqref{ieq:avv-2}. Since $v(R)=0$ and $\chi'(r)=2\chi_0(r-R)$, we have
\begin{align}
\int_R^\rho\eta\SN{v}^2\D r
\le\,& \epsilon^{-1}\int_R^\rho \eta\chi'\SN{v}^2 + \frac{\epsilon}{4}
\int_R^\rho\frac{\eta(r)}{\chi'(r)}\bigg|\int_R^r v'(t)\D t\bigg|^2\D r \notag \\
\le\,&  \epsilon^{-1}\int_R^\rho \eta\chi'\SN{v}^2\D r
+ C_2\epsilon \chi_0d^4 |\alpha(\rho)|^2 \int_R^\rho\xi\SN{v'}^2 \D r, \label{xi-v2}
\end{align}
where $\epsilon >0$ is a constant to be specified and $C_2$ is a positive constant independent of $m$, $\chi_0$, and $\rho$. Inserting \eqref{xi-v2} into \eqref{ReIm-am1}, we have
\begin{align}
\label{ieq:Am-21}
\Re \mathscr{A}_m(v,v) \ge\,& \Big[1-C_2\epsilon k_m^2 \chi_0d^4 |\alpha(\rho)|^2\Big] \int_R^\rho \xi\SN{v'}^2\D r
 -k_m^2\epsilon^{-1}\int_R^\rho \eta\sigma'\SN{v}^2\D r.
\end{align}
For \eqref{ReIm-am2}, noting that there exists a positive constant $C_3$ independent of $m$, $\chi_0$, and $\rho$ such that
\begin{align*}
\frac{\eta_1}{\eta\chi'} =\frac{r-\chi/\chi'}{r+r\chi'+ \chi+ 2\chi\chi'}\le C_3,
\end{align*}
we get
\begin{align}
\label{ieq:Am-22}
\Im \mathscr{A}_m(v,v)\le\,& -C_3k_m^2\int_R^\rho\eta\chi'\SN{v}^2\D r .
\end{align}
Then combining (\ref{ieq:Am-21}) and (\ref{ieq:Am-22}) yields
\begin{align}\label{ieq:Am-2}
& \Re \mathscr{A}_m(v,v)-2C_3^{-1}\epsilon^{-1}\Im \mathscr{A}_m(v,v) \nonumber\\
&\ge
    \Big[1-C_2\epsilon  k_m^2\chi_0d^4 |\alpha(\rho)|^2\Big]
    \int_R^\rho \xi\SN{v'}^2\D r  +\epsilon^{-1}k_m^2\int_R^\rho \eta\chi'\SN{v}^2\D r.
\end{align}
Then choosing $\epsilon^{-1}=2C_2  k_m^2\chi_0d^4 |\alpha(\rho)|^2$ gives
\eqref{ieq:avv-2}.

Finally, the well-posedness of problem \eqref{pro-wm} for all $m\ge 1$ follows directly from the estimates \eqref{ieq:avv-1}-\eqref{ieq:avv-2}. The proof is complete.
\end{proof}

Now we are ready to get the the well-posedness of problem \eqref{pro-hT} which will further leads to a continuity estimate for the DtN map $\hat T$.

\begin{theorem}\label{thm:exist-am}
There exists a unique solution $w\in\Hone[\Omega_\pml]$ to the problem \eqref{pro-hT}. Moreover, there exists a constant $C>0$ independent of $\chi_0$ and $\rho$ such that
\begin{align*}
\NHone[\Omega_\pml]{w}
\le Cd^6\SN{\alpha(\rho)}^4\SN{\beta(\rho)}^2
\N{\eta}_{H^{1/2}(\Gamma_R)}.
\end{align*}
\end{theorem}
\begin{proof}
There exists an extension $p\in \Hone[\Omega_{2R}\backslash\bar\Omega_R]$ which satisfies $p=\eta$ on $\Gamma_R$, $p=0$ on $\partial\Omega_{2R}\backslash\partial\Omega_{R}$, and
\begin{align}\label{trace-p}
\NHone[\Omega_{2R}\backslash\bar\Omega_R]{p}\le C\N{\eta}_{H^{1/2}(\Gamma_R)},
\end{align}
where the constant $C$ depends only on $R$ and $Z$. Now we extend $p$ by zero to $\Omega_\infty\backslash\bar\Omega_{2R}$ and designate the extension still by $p$.
A weak formulation of \eqref{pro-hT} is to find $\hat{w}:=w-p\in\zbHone[\Omega_\pml]$ such that
\begin{align*}
a_{\pml}(\hat{w},v) =-a_{\pml}(p,v) \quad
\forall\,v\in\zbHone[\Omega_\pml],
\end{align*}
where
\begin{align*}
a_{\pml}(\hat{w},v) := \int_{\Omega_\pml}\frac{1}{r\beta}\bigg(
\frac{1}{\alpha}\frac{\partial \hat{w}}{\partial r}\frac{\partial\bar{v}}{\partial r}
+\alpha \frac{\partial\hat{w}}{\partial z} \frac{\partial\bar{v}}{\partial z}
-\alpha k^2\hat{w}\bar{v}\bigg).
\end{align*}

Consider the Fourier series of $\hat{w}$ and $p$ with coefficients $\hat{w}_m$ and $p_m$, respectively which satisfy $\hat{w}_m=w_m-p_m$. Define
\begin{align*}
&W_1 = \sum_{m<kZ/\pi} \hat{w}_m\sin(m\pi z/Z),\quad
P_1 = \sum_{m<kZ/\pi} p_m\sin(m\pi z/Z),\\
&W_2 = \sum_{m>kZ/\pi} \hat{w}_m\sin(m\pi z/Z),\quad
P_2 =  \sum_{m>kZ/\pi} p_m\sin(m\pi z/Z).
\end{align*}
It is clear that $\hat{w}=W_1+W_2$, $p=P_1+P_2$, and
\begin{align}\label{W1W2}
a_{\pml}(W_1,W_1) =-a_{\pml}(P_1,W_1), \quad
a_{\pml}(W_2,W_2) =-a_{\pml}(P_2,W_2).
\end{align}
Let the function $\xi(r)$ and the constant $C_1$ be given in Lemma~\ref{lem:avv}, and let the space $V(\Omega_\pml)$, as well as the equipped norm, be defined analogous to $V_R$. Using \eqref{ieq:avv-2} and $\hat{w}_m(R)=0$, we have
\begin{align*}
&\N{W_1}_{V(\Omega_\pml)}^2  \\
\le\,& C\sum_{m<kZ/\pi} \bigg[\int_R^\rho\frac{1}{r}
\SN{\hat w_m'(r)}^2\D r
+(k^2+m^2\pi^2/Z^2)\int_R^\rho\frac{1}{r}\bigg|\int_R^r\hat w_m'(t)\D t\bigg|^2\D r\bigg] \\
\le\,& Cd^3\SN{\alpha(\rho)}^2\SN{\beta(\rho)}^2
\sum_{m<kZ/\pi}  \int_R^\rho\xi(r)\SN{\hat w_m'(r)}^2\D r \\
\le\,& Cd^3\SN{\alpha(\rho)}^2\SN{\beta(\rho)}^2
\sum_{m<kZ/\pi}  \big[\Re \mathscr{A}_m(\hat{w}_m,\hat{w}_m)
-C_1d^3\SN{\alpha(\rho)}^{2}\Im \mathscr{A}_m(\hat{w}_m,\hat{w}_m)\big].
\end{align*}
Using \eqref{W1W2} and the relation $a_{\pml}(W_1,W_1) =\displaystyle \frac{Z}{2}\sum\limits_{m<kZ/\pi} \mathscr{A}_m(\hat{w}_m,\hat{w}_m)$, we deduce that
\begin{align*}
\N{W_1}_{V(\Omega_\pml)}^2
\le\,& Cd^6\SN{\alpha(\rho)}^4\SN{\beta(\rho)}^2 \SN{a_{\pml}(P_1,W_1)}.
\end{align*}
Since $P_1$ is only supported in $\Omega_{2R}\backslash\ol\Omega_R$, using \eqref{trace-p} and Schwartz's inequality, we easily get
\begin{align}\label{est:W1}
\N{W_1}_{V(\Omega_\pml)}  \le Cd^6\SN{\alpha(\rho)}^4\SN{\beta(\rho)}^2
\N{P_1}_{V(\Omega_{2R}\backslash\ol\Omega_R)}
\le Cd^6\SN{\alpha(\rho)}^4\SN{\beta(\rho)}^2
\N{\eta}_{H^{1/2}(\Gamma_R)}.
\end{align}

The estimate of $W_2$ is similar but easier. Since $k_m^2<0$ for $m>kZ/\pi$, we have
\begin{align*}
\N{W_2}_{V(\Omega_\pml)}^2  \le\,& C\sum_{m> kZ/\pi} \int_R^\rho\frac{1}{r}
\Big[\SN{\hat w_m'(r)}^2- k_m^2\SN{\hat w_m(r)}^2\Big]\D r \\
\le\,& Cd\SN{\alpha(\rho)}^2\SN{\beta(\rho)}^2
\sum_{m>kZ/\pi}  \int_R^\rho \Big[\xi\SN{\hat{w}_m'(r)}^2
    -C_0 k_m^2 \eta\SN{\hat{w}_m(r)}^2\Big]\D r \\
\le\,& Cd\SN{\alpha(\rho)}^2\SN{\beta(\rho)}^2
\sum_{m>kZ/\pi} \Re \mathscr{A}_m(\hat{w}_m,\hat{w}_m) .
\end{align*}
Analogously, using \eqref{W1W2} and the relation $a_{\pml}(W_2,W_2) =\displaystyle \frac{Z}{2}\sum\limits_{m>kZ/\pi} \mathscr{A}_m(\hat{w}_m,\hat{w}_m)$, we deduce that
\begin{align}\label{est:W2}
\N{W_2}_{V(\Omega_\pml)}
\le C d\SN{\alpha(\rho)}^2\SN{\beta(\rho)}^2
\N{\eta}_{H^{1/2}(\Gamma_R)}.
\end{align}
The proof is finished by combining \eqref{est:W1}--\eqref{est:W2} and \eqref{trace-p}.
\end{proof}

\begin{corollary}
Let $\hat{T}$ be the approximate DtN operator defined in \eqref{hT}.
There exists a constant $C>0$ independent of $\chi_0$ and $\rho$ such that
\begin{align*}
\big\|\hat{T}\eta\big\|_{H^{-1/2}(\Gamma_R)}\le Cd^6\SN{\alpha(\rho)}^4\SN{\beta(\rho)}^2 \N{\eta}_{H^{1/2}(\Gamma_R)}
\quad \forall\,\eta\in H^{1/2}_0(\Gamma_R).
\end{align*}
\end{corollary}
\begin{proof}
Given $\eta\in H^{1/2}_0(\Gamma_R)$, let $w\in\Hone[\Omega_\pml]$ be the solution to \eqref{pro-hT}.
Define $\displaystyle \Bv=\Big(\frac{1}{r\alpha\beta}\frac{\partial w}{\partial r},
\frac{\alpha }{r\beta}\frac{\partial w}{\partial z}\Big)$. From \eqref{eqn-hT} we know that
\begin{align}\label{eq:div-Bv}
\Div_{r,z}\Bv = \frac{\partial}{\partial r}
\bigg(\frac{1}{r\alpha\beta}\frac{\partial w}{\partial r}\bigg)
+\frac{\partial}{\partial z}\bigg(\frac{\alpha}{r\beta}\frac{\partial w}{\partial z}\bigg)
=-k^2\frac{\alpha}{r\beta} w \in \Ltwo[\Omega_\pml].
\end{align}
Let $\Bn=(1,0)$ denote the outer normal on $\Gamma_R$. Then \eqref{hT} and the trace theorem indicate that
\begin{align*}
\big\|\hat{T}\eta\big\|_{H^{-1/2}(\Gamma_R)}
=\,& R^2\N{\frac{\partial}{\partial r}(r^{-1}w)}_{H^{-1/2}(\Gamma_R)}  \\
\le\,&  R^2 \N{\frac{1}{r\alpha\beta}\frac{\partial w}{\partial r}}_{H^{-1/2}(\Gamma_R)}
+\N{w}_{H^{-1/2}(\Gamma_R)}\\
=\,& R^2\N{\Bv\cdot\Bn}_{H^{-1/2}(\Gamma_R)} +\N{w}_{H^{-1/2}(\Gamma_R)} \\
\le\,&C\left(\NHdiv[\Omega_\pml]{\Bv} + \N{w}_{H^1(\Omega_\pml)}\right).
\end{align*}
Together with \eqref{eq:div-Bv}, this yields $\big\|\hat{T}\eta\big\|_{H^{-1/2}(\Gamma_R)}
\le C\N{w}_{H^1(\Omega_\pml)}$. The proof is finished upon using Theorem~\ref{thm:exist-am}.
\end{proof}

\subsection{Estimation of $T-\hat{T}$}

Before studying the convergence of the approximate solution $\hat u$ of the problem \eqref{model-hR}, the error estimate of $T\eta -\hat{T}\eta$ for any $\eta\in H^{1/2}_0(\Gamma_R)$ will be addressed in this subsection. To do this, we define the wave propagation operator $\Cp$ which extends $\eta$ to the exterior of $\Omega_R$ in the following way
\begin{align*}
\Cp(\eta)(r,z) :=\frac{r}{R}\sum_{m=1}^\infty \frac{H^{(1)}_1(k_mr)}{H^{(1)}_1(k_mR)}
    \eta_m\sin(m\pi z/Z) \quad \forall\, r\ge R.
\end{align*}
In view of \eqref{exp-u}, it is easy to see that $\Cp(\eta)$ is the solution to the scattering problem
\begin{align*}
&r\frac{\partial}{\partial r}\bigg(\frac{1}{r}\frac{\partial}{\partial r}\Cp(\eta)\bigg)
+\frac{\partial^2}{\partial z^2}\Cp(\eta) + k^2\Cp(\eta)= 0
\quad \hbox{in}\;\; \Omega\backslash\overline{\Omega_R},  \\
&\Cp(\eta)=\eta \quad \hbox{on}\;\;\Gamma_R,  \\
&\lim_{r\to\infty}\SN{\frac{\partial}{\partial r}\Cp(\eta) - ik\Cp(\eta)} = 0.
\end{align*}
In particular, the exact solution $u$ to the scattering problem \eqref{eq:modelR} satisfies $\Cp(u|_{\Gamma_R}) = u$ in $\Omega\backslash\bar\Omega_R$.

Using the complex stretching, we also define a modified wave propagation operator as
\begin{align}\label{eq:tp}
\tilde\Cp(\eta)(r,z) :=\frac{\tilde r}{R}\sum_{m=1}^\infty \frac{H^{(1)}_1(k_m\tilde r)}{H^{(1)}_1(k_mR)}
    \eta_m\sin(m\pi z/Z) \quad \forall\, r\ge R.
\end{align}
It is clear that $\tilde\Cp(\eta)(r,z)=\Cp(\eta)(\tilde r,z)$. The chain rule indicates that $\tilde\Cp(\eta)$ satisfies
\begin{align*}
&\frac{r\beta}{\alpha}\frac{\partial}{\partial r}\bigg(\frac{1}{r\alpha\beta}
    \frac{\partial}{\partial r}\tilde\Cp(\eta)\bigg)
+\frac{\partial^2}{\partial z^2}\tilde\Cp(\eta) + k^2\tilde\Cp(\eta)= 0
\quad \hbox{in}\;\; \Omega\backslash\overline{\Omega_R},  \\
&\tilde\Cp(\eta)=\eta \quad \hbox{on}\;\;\Gamma_R .
\end{align*}

\begin{lemma}\label{lem:decay-tP}
There exists a constant $C>0$ independent of $\rho$ and $\chi_0$ such that
\begin{align*}
\big\|\tilde\Cp(\eta)\big\|_{H^{1/2}(\Gamma_\rho)}
\le CR^{-1}\big|\tilde\rho\big| e^{-0.8\kappa\chi_0d^2}
\|\eta\|_{H^{1/2}(\Gamma_R)}.
\end{align*}
\end{lemma}
\begin{proof}
We recall \cite[Lemma~2.2]{che05} for the following estimate
\begin{align*}
\SN{H^{(1)}_1(z)} \le e^{-(1-t^2/|z|^2)^{1/2}\Im(z)}\SN{H^{(1)}_1(t)},
\quad 0< t\le |z|, \quad 0\le \arg(z) \le \pi/2.
\end{align*}
For $1\le m< kZ/\pi$, we have $k_m>0$ and $0<\arg\tilde\rho<\pi/2$. It follows from $R/|\tilde\rho|\le 0.5$ that
\begin{align}\label{ieq:Hm-small}
\big|H^{(1)}_1(k_m\tilde\rho)\big|
\le e^{-0.8\kappa\chi_0d^2}\big|H^{(1)}_1(k_mR)\big|.
\end{align}
For $m> kZ/\pi$, we have $k_m=\Vi\SN{k_m}$, $H^{(1)}_1(k_m\tilde\rho)$ is connected with the modified Bessel function through
\ben
H^{(1)}_1(k_m\tilde\rho) = -\frac{2}{\pi} K_1(|k_m|\tilde\rho).
\een
Moreover, by \cite[eq. (10.32.9)]{olv10}, the modified Bessel function of the $m^{\rm th}$ order satisfies
\ben
K_1(|k_m|\tilde\rho) =\int_0^{\infty}e^{-|k_m|\tilde\rho\cosh t}\cosh(t)\D t.
\een
We infer that
\begin{align}
\big|H^{(1)}_1(k_m\tilde\rho)\big| \le\,& \frac{2}{\pi} \int_0^{\infty}
    e^{-|k_m|(\Re\tilde\rho)\cosh t}\cosh(t)\D t \le\, e^{-\kappa\chi_0d^2}\big|H^{(1)}_1(k_mR)\big|.
    \label{ieq:Hm-large}
\end{align}

Substituting \eqref{ieq:Hm-small} and \eqref{ieq:Hm-large} into \eqref{eq:tp} shows
\begin{align*}
\big\|\tilde\Cp(\eta)\big\|_{H^{1/2}(\Gamma_\rho)}^2 =\,&
\frac{\big|\tilde\rho\big|^2}{R^{2}}\sum_{m=1}^\infty m
\bigg|\frac{H^{(1)}_1(k_m\tilde r)}{H^{(1)}_1(k_mR)}\bigg|^2\SN{\eta_m}^2
\le \frac{\big|\tilde\rho\big|^2}{R^{2}} e^{-1.6\kappa\chi_0d^2}
\|\eta\|_{H^{1/2}(\Gamma_R)}^2.
\end{align*}
The proof is finished.
\end{proof}

The error estimates of $T-\hat{T}$ is given in the following corollary.
\begin{corollary}\label{cor:T-hatT}
For any $\eta\in H^{1/2}(\Gamma_R)$, there exists a constant $C>0$ independent of $\rho$ and $\chi_0$ such that
\begin{align}\label{result:T-hatT}
\|T\eta -\hat{T}\eta\|_{H^{-1/2}(\Gamma_R)}\le Cd^6\SN{\alpha(\rho)}^4\SN{\beta(\rho)}^2\big|\tilde\rho\big| e^{-0.8\kappa\chi_0d^2}
\|\eta\|_{H^{1/2}(\Gamma_R)}.
\end{align}
\end{corollary}
\begin{proof}
To prove this result, we consider the following problem in the PML region
\begin{subequations}\label{AP:errT}
\begin{align}
&\frac{r\beta}{\alpha}\frac{\partial}{\partial r}
\bigg(\frac{1}{r\alpha\beta}\frac{\partial p}{\partial r}\bigg)
+\frac{\partial^2p}{\partial z^2} + k^2 p = 0
\quad \hbox{in}\;\;\Omega_{\pml}, \label{eqn-errT} \\
&p=0 \quad \hbox{on}\;\;\Gamma_R, \label{bc0-errT}\\
&p=0 \quad \hbox{on}\;\;\{z=0\}\cup\{z=Z\}, \label{bc1-errT}\\
&p=w \quad \hbox{on}\;\;\Gamma_\rho.\label{bc1-errT2}
\end{align}
\end{subequations}
It follows analogously from the proof of Theorem~\ref{thm:exist-am} that there exists a constant $C>0$ independent of $\rho$ and $\chi_0$ such that
\begin{align*}
\NHone[\Omega_\pml]{p}
\le Cd^6\SN{\alpha(\rho)}^4\SN{\beta(\rho)}^2
\N{w}_{H^{1/2}(\Gamma_\rho)}.
\end{align*}
For any $\varphi\in H^1(\Omega_{\mathrm{pml}})$ such that $\varphi=0$ on $\Gamma_\rho$, we obtain
\begin{align*}
a_{\pml}(p,\varphi)+ \frac{1}{R}\int_{\Gamma_R}\frac{\partial p}{\partial r} \bar\varphi dz=0.
\end{align*}
Thus,
\begin{align}\label{Dp:err}
\left\| \frac{\partial p}{\partial r}\right\|_{H^{-1/2}(\Gamma_R)}\le C_1d^6\SN{\alpha(\rho)}^4\SN{\beta(\rho)}^2
\N{w}_{H^{1/2}(\Gamma_\rho)},
\end{align}
where $C_1>0$ is a constant independent of $\rho$ and $\chi_0$. Now, for any $\eta\in H^{1/2}(\Gamma_R)$, it follows from the definitions of $T$ and $\hat{T}$ that
\begin{align}\label{redef}
T\eta -\hat{T}\eta = \displaystyle R^2\frac{\partial}{\partial r}(r^{-1}p) \quad\mbox{on}\quad\Gamma_R,
\end{align}
where $p$ is the solution to the Dirichlet problem (\ref{AP:errT}) with $w=\tilde\Cp(\eta)$ on $\Gamma_\rho$.
Then the final estimates (\ref{result:T-hatT}) results by combining (\ref{Dp:err}) and (\ref{redef}).
\end{proof}

\subsection{Exponential convergence}
\label{sec:expcon}

Now we are ready to show the convergence of the solution $\hat{u}$ to the approximate problem (\ref{pro-pml}) , or equivalently, problem (\ref{model-hR}). Similar to \eqref{weak}, the variational problem of \eqref{model-hR} is proposed as follows: find $\hat{u}\in V_R$ such that $\hat{u}=u_D$ on $\Sigma_R$ and
\begin{align}\label{weak-app}
\hat{a}(\hat{u},v) = \int_{\Gamma_{\mathrm{right}}^1} u_N\bar{v}\D r \quad \forall\, v\in V_{R,0},
\end{align}
where the sesquilinear form $\hat{a}$ is defined by
\begin{align}\label{eq:hat-a}
\hat{a}(u,v) = \int_{\Omega_R}\bigg[\frac{\partial u}{\partial r}
\frac{\partial\bar{v}}{\partial r}
+\frac{\partial u}{\partial z}\frac{\partial \bar{v}}{\partial z}
- (k^2 + \Vi\omega\mu\sigma) u \bar{v}\bigg]\frac{1}{r}\, \D r\D z
-\frac{1}{R^2}\langle \hat Tu + u, v\rangle_{\Gamma_R}.
\end{align}

\begin{theorem}\label{thm:err}
Assuming the PML parameters $\rho$ and $\chi_0$ being such that
\begin{align}\label{well:assum}
C_0:=C_{\mathrm{inf}}-C_1d^6\SN{\alpha(\rho)}^4\SN{\beta(\rho)}^2\big|\tilde\rho\big| e^{-0.8\kappa\chi_0d^2}>0,
\end{align}
where the constants $C_{\mathrm{inf}}>0, C_1>0$ independent of $\rho$ and $\chi_0$ come from Theorem~\ref{thm:infsup} and Corollary~\ref{cor:T-hatT}, then there exists a unique solution to the variational problem (\ref{weak-app}). Moreover, there exists a constant $C>0$ independent of $\rho$ and $\chi_0$ such that
\begin{align}\label{eq:err}
\N{u-\hat{u}}_{V_R} \le Cd^6\SN{\alpha(\rho)}^4\SN{\beta(\rho)}^2\big|\tilde\rho\big| e^{-0.8\kappa\chi_0d^2} \N{\hat{u}}_{V_R}.
\end{align}
\end{theorem}
\begin{proof}
It follows from Theorem~\ref{thm:infsup} and Corollary~\ref{cor:T-hatT} that there exists constants $C_{\mathrm{inf}}>0, C_1>0$ independent of $\rho$ and $\sigma_0$ such that
\begin{align*}
\sup_{0\ne v\in V_{R,0}}\frac{\SN{\hat{a}(u,v)}}{\N{v}_{V_R}}=\,& \sup_{0\ne v\in V_{R,0}}\frac{\SN{a(u,v)+ R^{-2}\langle (T-\hat T)u, v\rangle_{\Gamma_R}}}{\N{v}_{V_R}}\\
\ge \,& \sup_{0\ne v\in V_{R,0}}\frac{\SN{a(u,v)}}{\N{v}_{V_R}} -R^{-2}\sup_{0\ne v\in V_{R,0}}\frac{\SN{\langle (T-\hat T)u, v\rangle_{\Gamma_R}}}{\N{v}_{V_R}}\\
\ge \,& \left(C_{\mathrm{inf}}-C_1d^6\SN{\alpha(\rho)}^4\SN{\beta(\rho)}^2\big|\tilde\rho\big| e^{-0.8\kappa\chi_0d^2} \right)\N{u}_{V_R}.
\end{align*}
Then under the assumption (\ref{well:assum}), the existence and uniqueness of the solution to the variational problem (\ref{weak-app}) follows immediately. It remains to prove the error estimate \eqref{eq:err}. From \eqref{weak} and \eqref{weak-app}, it is easy to see that the error function $e:= u-\hat u\in V_{R,0}$ satisfies
\begin{align}\label{weak-app1}
a(e,v) = R^{-2}\langle (T-\hat T)\hat{u}, v\rangle_{\Gamma_R}\quad \forall\, v\in V_{R,0}.
\end{align}
By the inf-sup condition in Theorem~\ref{thm:infsup}, we have
\begin{align*}
\N{u-\hat{u}}_{V_R}\le\,& C^{-1}_{\inf}\sup_{0\ne v\in V_{R,0}}\frac{\SN{a(e,v)}}{\N{v}_{V_R}}\\
=\,& R^{-2}C^{-1}_{\inf}\sup_{0\ne v\in V_{R,0}}\frac{\SN{\langle (T-\hat T)\hat{u}, v\rangle_{\Gamma_R}}}{\N{v}_{V_R}} \\
\le\,& Cd^6\SN{\alpha(\rho)}^4\SN{\beta(\rho)}^2\big|\tilde\rho\big| e^{-0.8\kappa\chi_0d^2} \N{\hat{u}}_{V_R},
\end{align*}
where $C>0$ is a constant independent of $\rho$ and $\chi_0$.
\end{proof}

\section{Numerical experiments}
\label{sec:numer}

In this section, two numerical examples are presented to illustrate the efficiency of our model to simulate the signal propagation in axons. All the parameters are selected in dimensionless type. The finite element method is utilized for the numerical discretization, for which the error estimates is left for future work, and the particular implementation for the numerical experiments is programmed in Matlab.

In the first example, we test the convergence of the numerical solution arising from the PML truncation and finite element discretization. We consider the following PML problem
\begin{subequations}
\begin{align*}
&\frac{r\beta}{\alpha}\frac{\partial}{\partial r}
\bigg(\frac{1}{r\alpha\beta}\frac{\partial w}{\partial r}\bigg)
+\frac{\partial^2w}{\partial z^2} + k^2 w = 0
\quad \hbox{in}\;\;\Omega', \\
&w=\eta \quad \hbox{on}\;\;\Gamma_{R'}, \\
&w=0 \quad \hbox{on}\;\;\partial\Omega'\backslash\Gamma_{R'},
\end{align*}
\end{subequations}
where $\Omega' = [0,\pi]\times[1,11]$ with the PML region $\Omega_{\rm pml} = [0,\pi]\times[10,11]$ and $R=10$, $R'=11$. We set $k=2$, $\chi_0=40$ and the exact solution is given by $w=rH_1^{(1)}(k_mr)\sin(mz)$ with $m=1$. Figure~\ref{example1} displays the numerical errors in $L^2$ and $H^1$-norms with respect to the finite element meshsize $h$ which clearly shows the second- and first-order convergence, respectively.

\begin{figure}[htbp]
\centering
\includegraphics[scale=0.1]{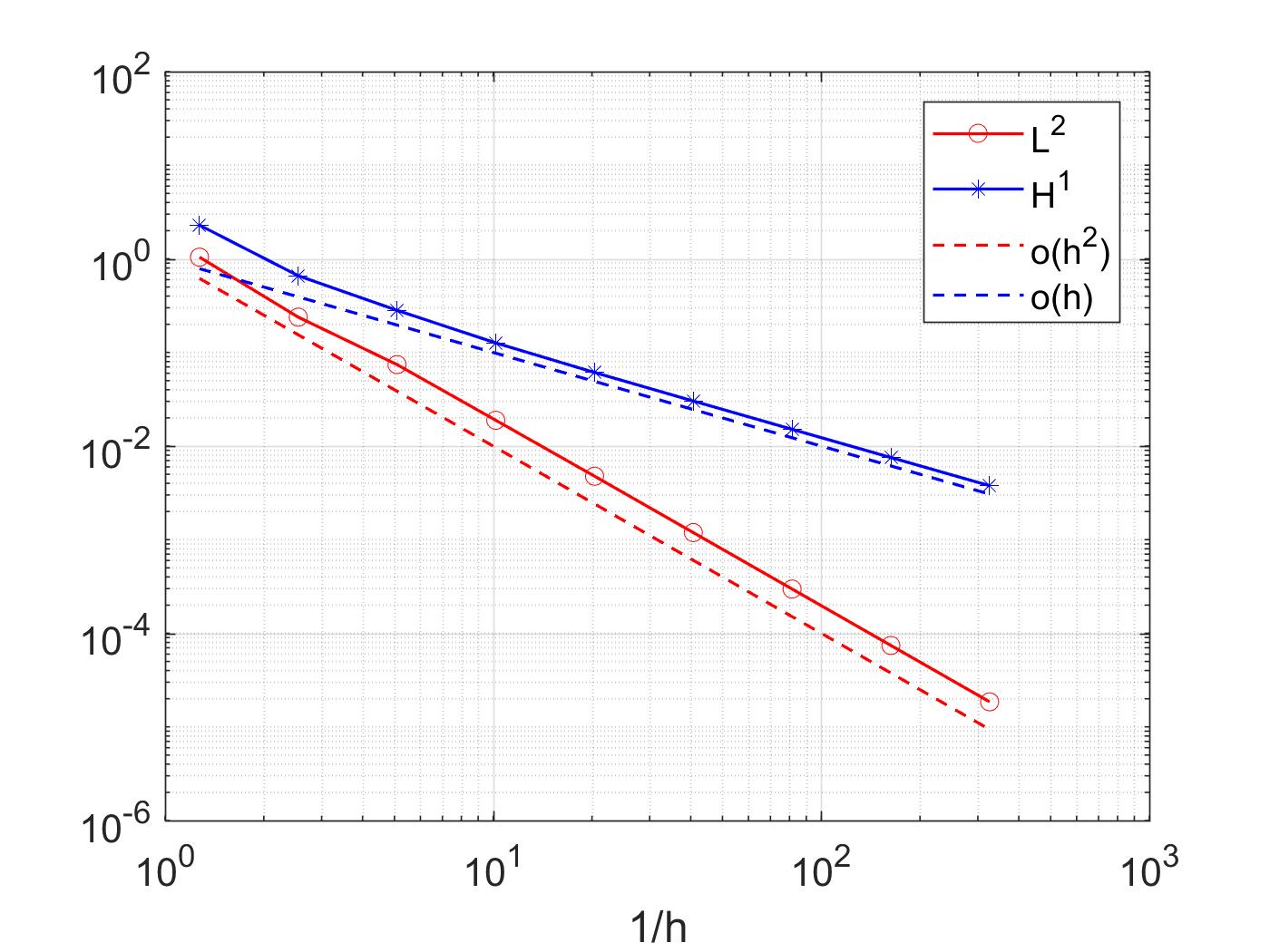}
\caption{Log-log plot the numerical errors in $L^2$ and $H^1$-norms with respect to the finite element meshsize $h$.}
\label{example1}
\end{figure}

Next, we test the propagation of an incident signal given on $\Gamma_{\rm left}$ in the axon and the TE mode is considered. We set
\ben
u_1(r) = \begin{cases}
-J_{1}(k_cr) & \mbox{on}\quad\Gamma_{\rm left}\cap\partial D_1,\cr
0 & \mathrm{otherwise},
\end{cases}\quad u_2(r)=0,\quad u_N=0,
\een
with $k_c = 2\times 3.831705970207512$ and choose the parameters
\ben
(\varepsilon,\sigma)=\begin{cases}
(2,0.2) & \mbox{in}\quad D_1,\cr
(10,0) & \mbox{in}\quad D_2,\cr
(1.2,0) & \mbox{in}\quad D^c,
\end{cases}\quad \omega=5,\quad \mu=1.
\een
The considered axon structures, wherein the axon is wrapped by a long myelin sheath or two separated myelin sheaths or the myelin sheath is absent, are presented in Figure \ref{fig:setup}. The real parts of the wave fields $H_\theta$ and $E_z, E_r$ are plotted in Figures \ref{fig:Ht}-\ref{fig:Er}, respectively. It can be clearly observed that the existence of myelin sheath can gather the electromagnetic fields to propagate mainly in myelin sheath.

\begin{figure}[htbp]
\centering
\begin{tabular}{ccc}
\includegraphics[scale=0.1]{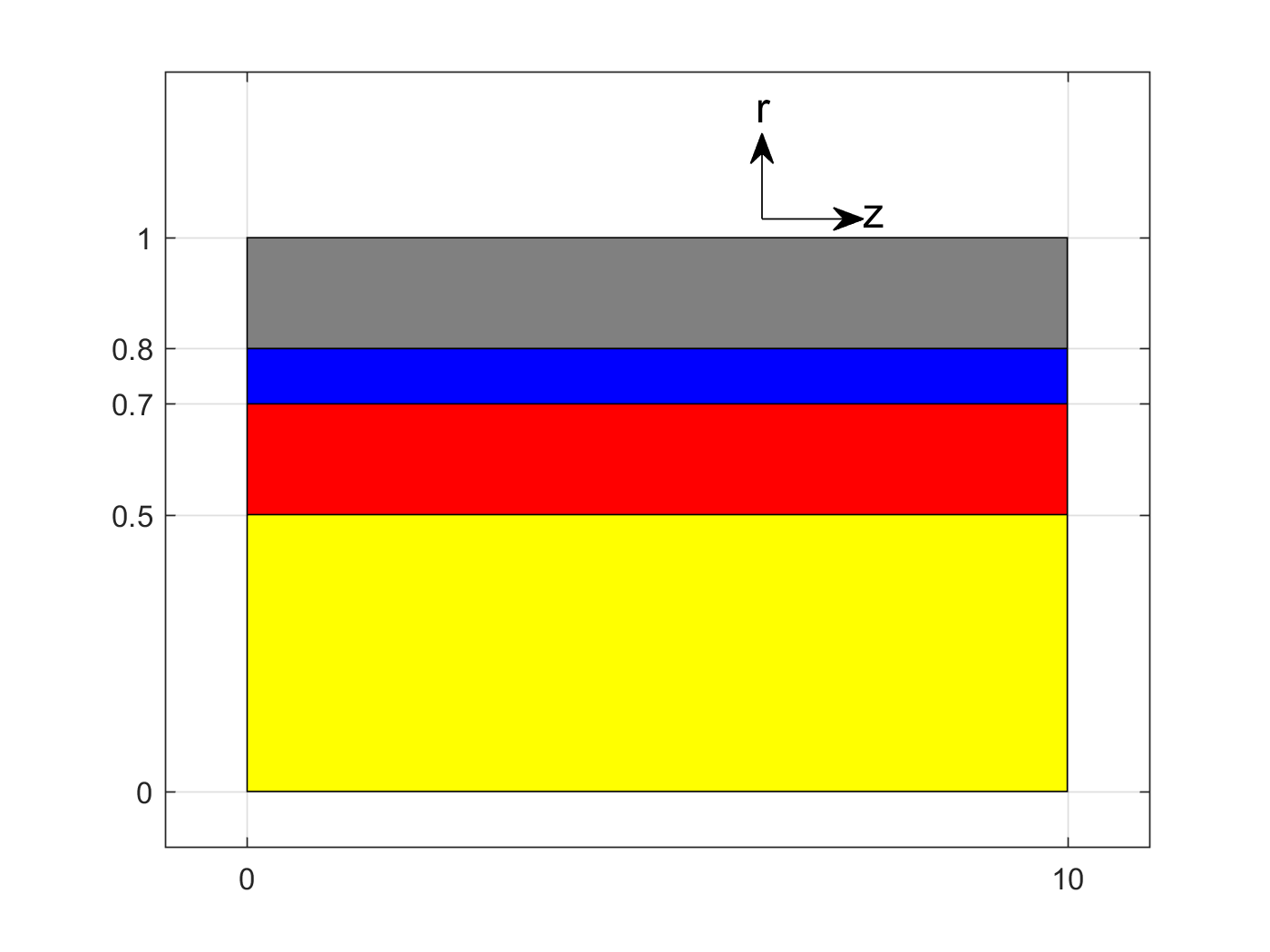} &
\includegraphics[scale=0.1]{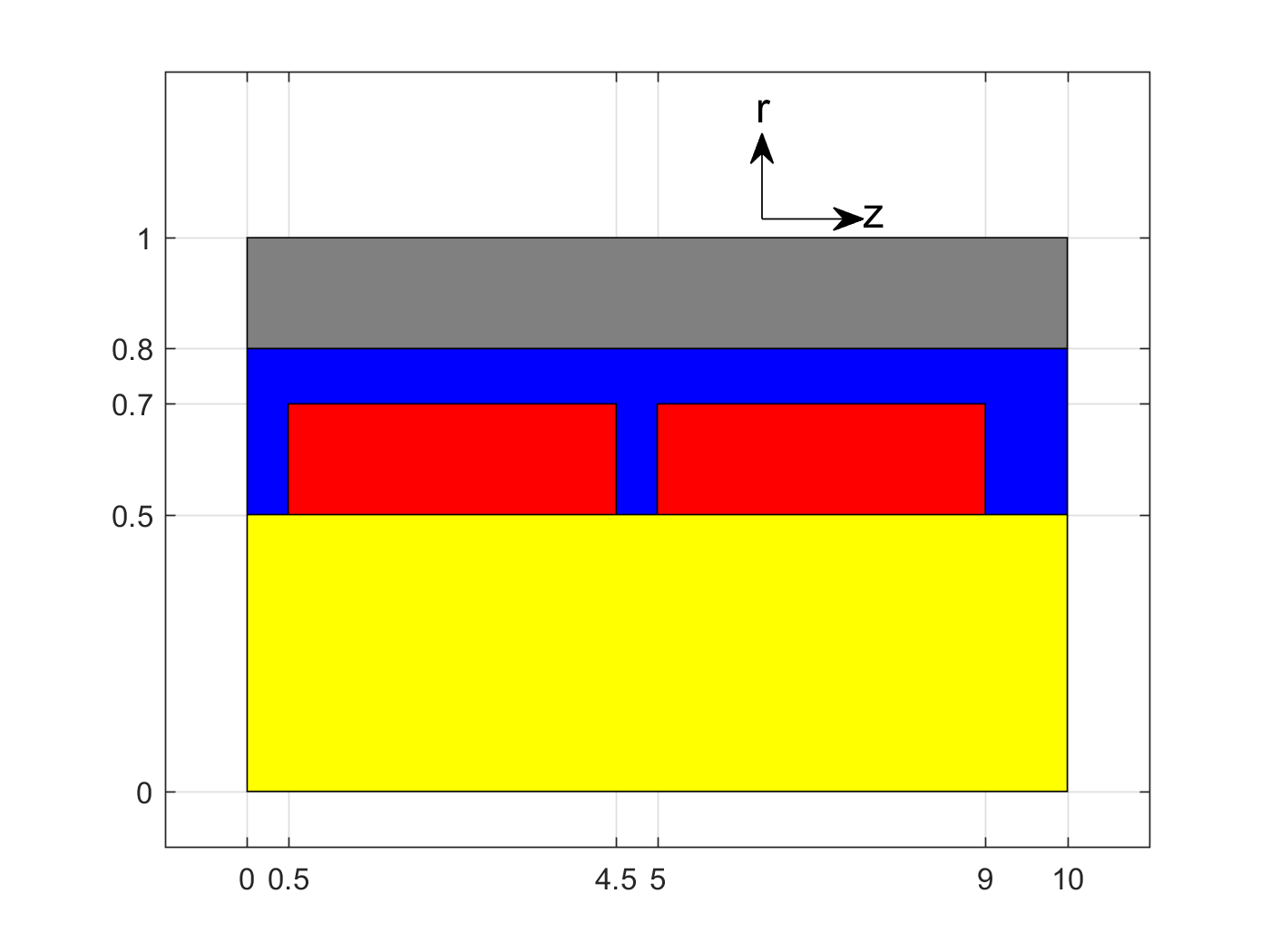} &
\includegraphics[scale=0.1]{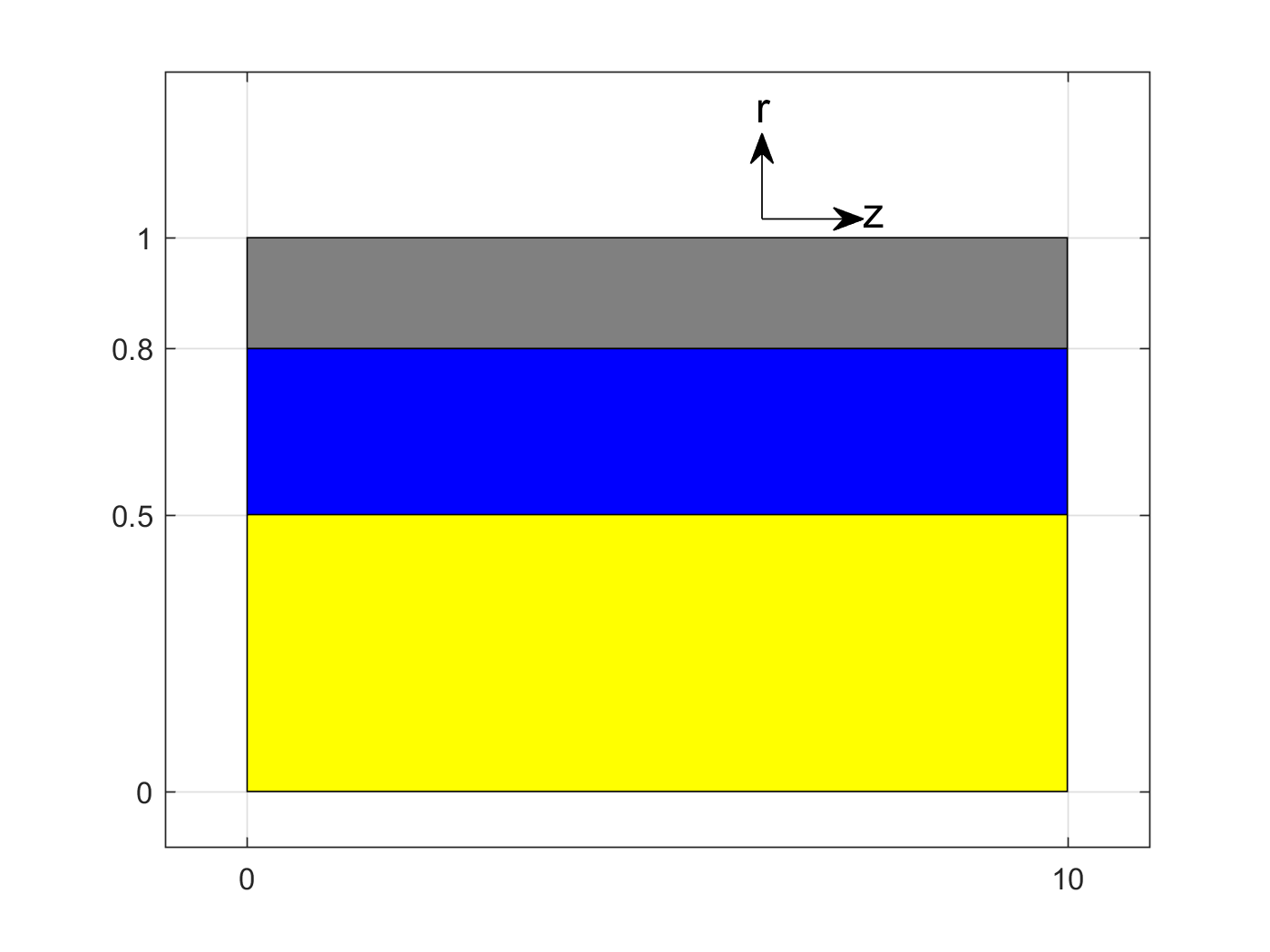}\\
(a) & (b) & (c)
\end{tabular}
\caption{The set-up for simulations wherein the axon is coated by a long myelin sheath (a) or two separated myelin sheaths (b), or the myelin sheath is absent (c). The yellow, red, blue and grey zones represent the regions $D_1$, $D_2$, $D_R^c$ and $\Omega_{\mathrm{pml}}$, respectively.}
\label{fig:setup}
\end{figure}

\begin{figure}[htbp]
\centering
\begin{tabular}{ccc}
\includegraphics[scale=0.1]{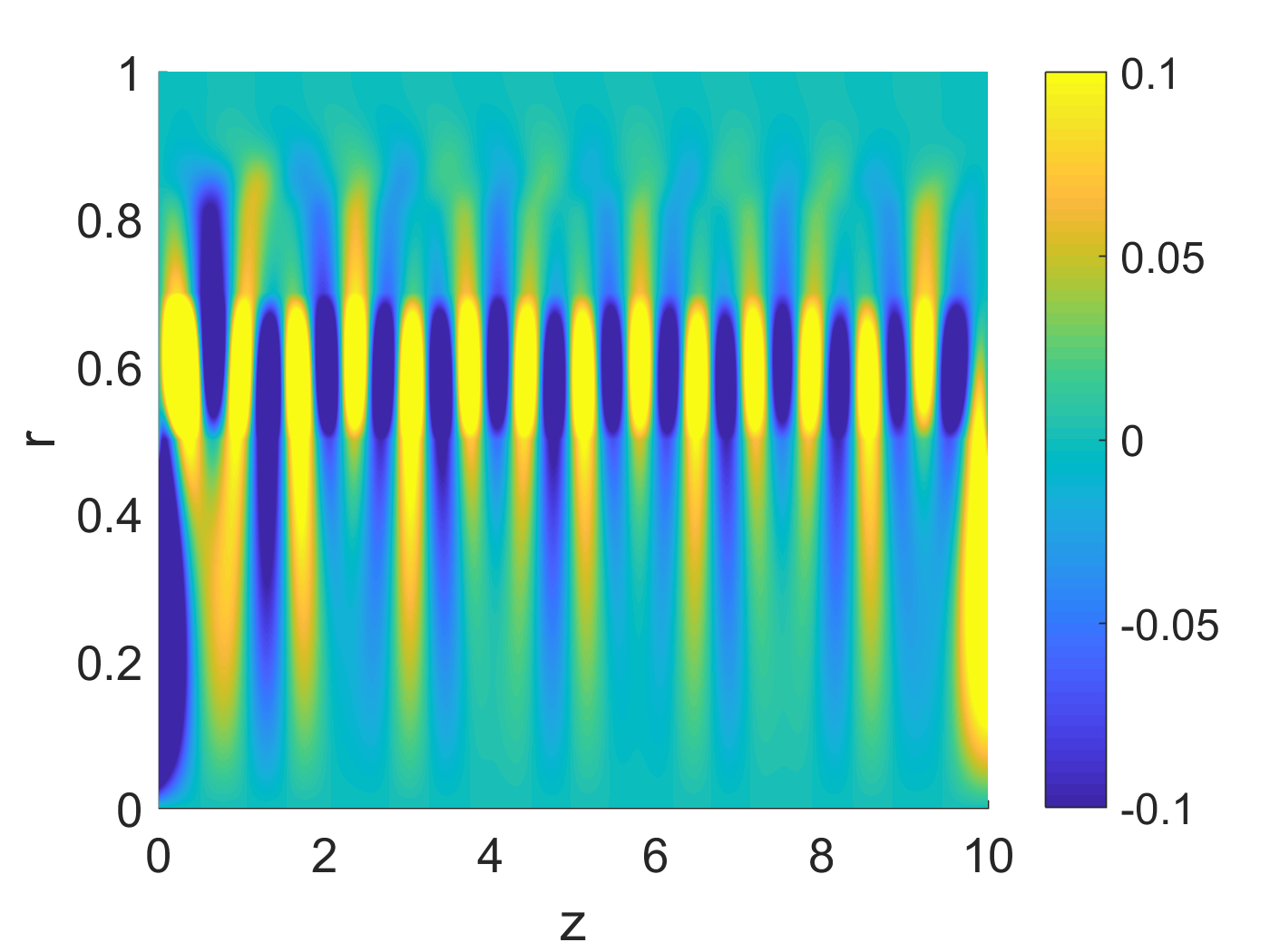} &
\includegraphics[scale=0.1]{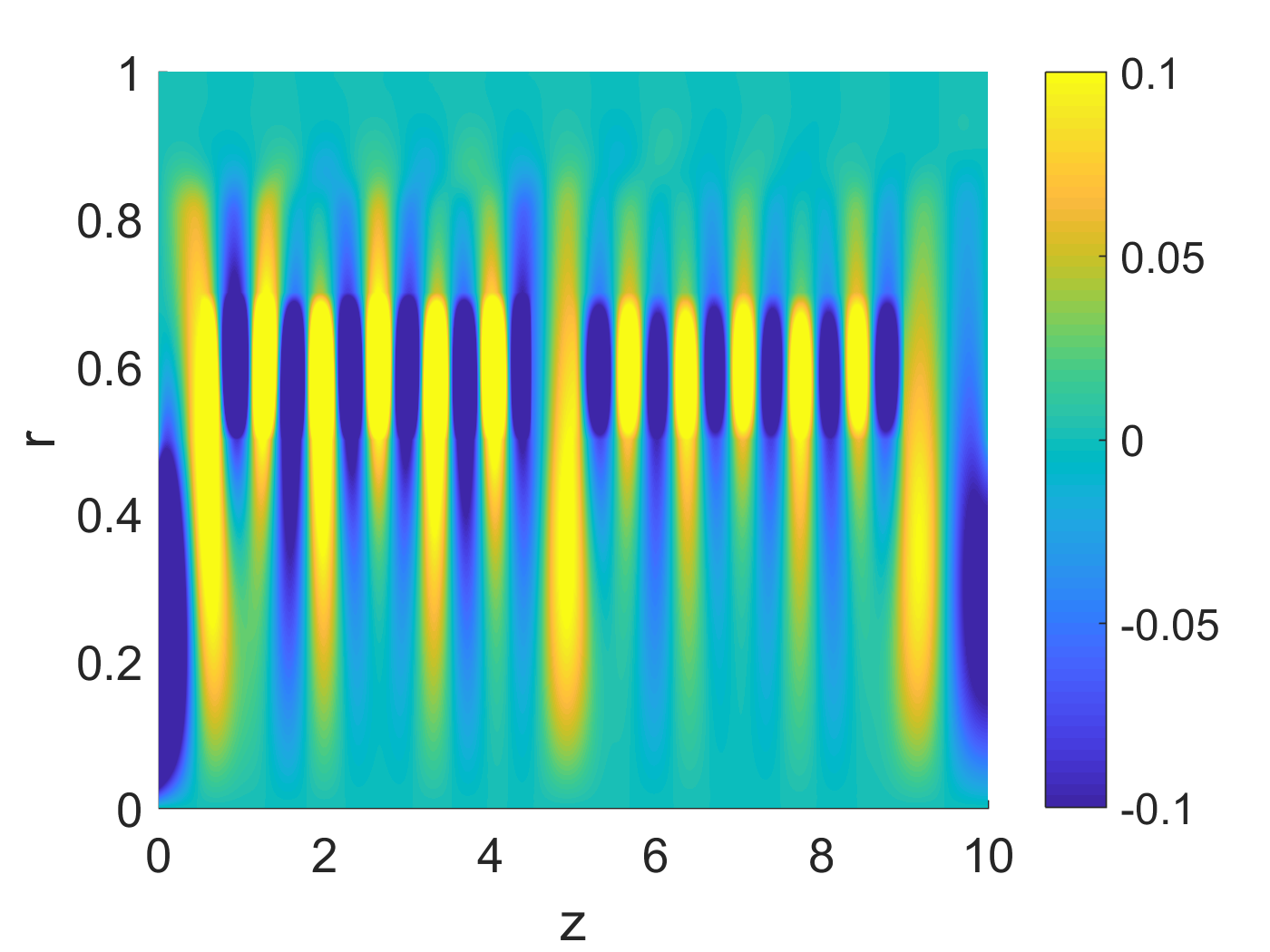} &
\includegraphics[scale=0.1]{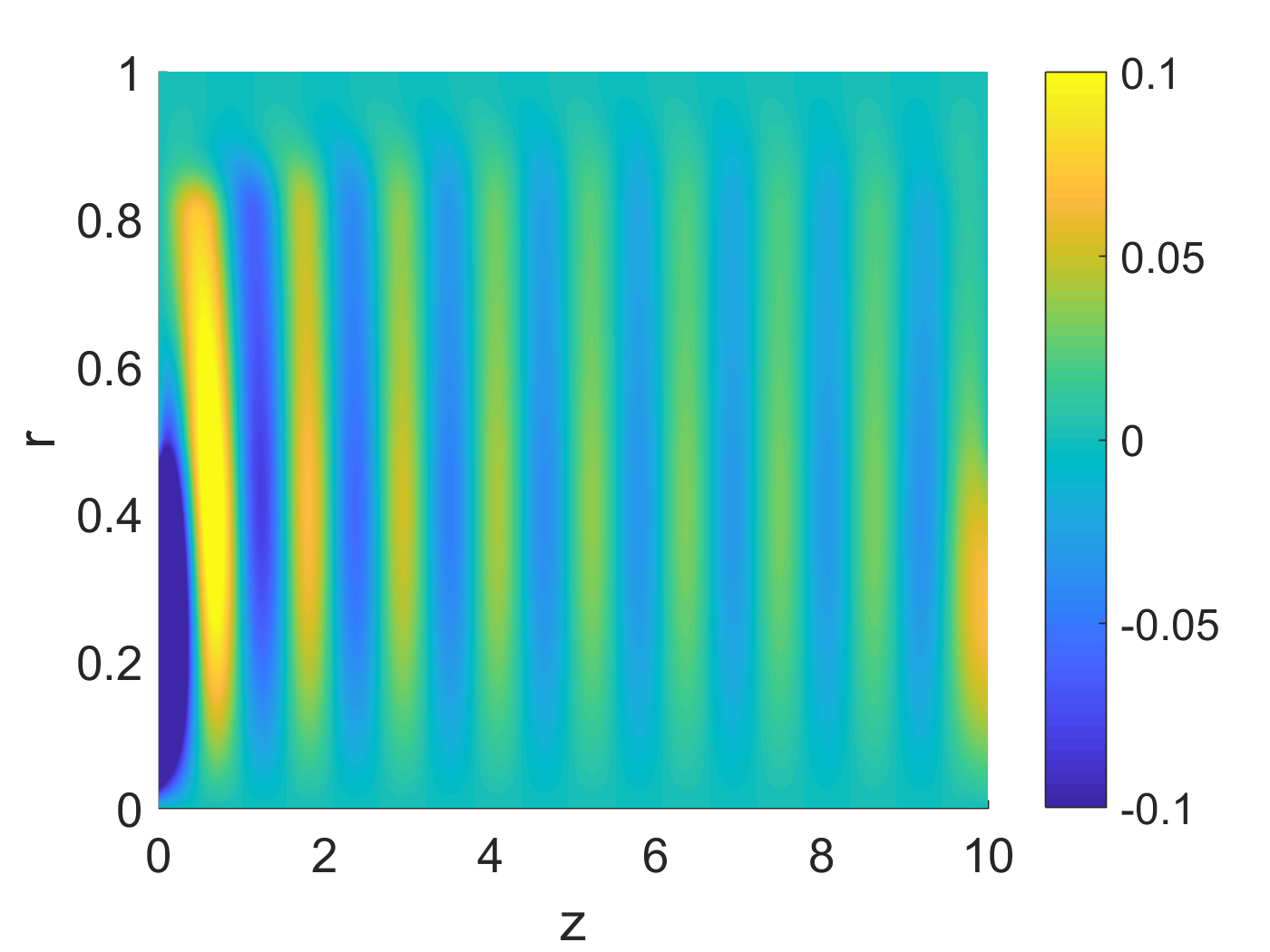}
\end{tabular}
\caption{The real part of the magnetic field $H_\theta$ for the considered axon structure shown in Figure~\ref{fig:setup}. Left: Figure~\ref{fig:setup}(a), middle: Figure~\ref{fig:setup}(b), right: Figure~\ref{fig:setup}(c).}
\label{fig:Ht}
\end{figure}

\begin{figure}[htbp]
\centering
\begin{tabular}{ccc}
\includegraphics[scale=0.1]{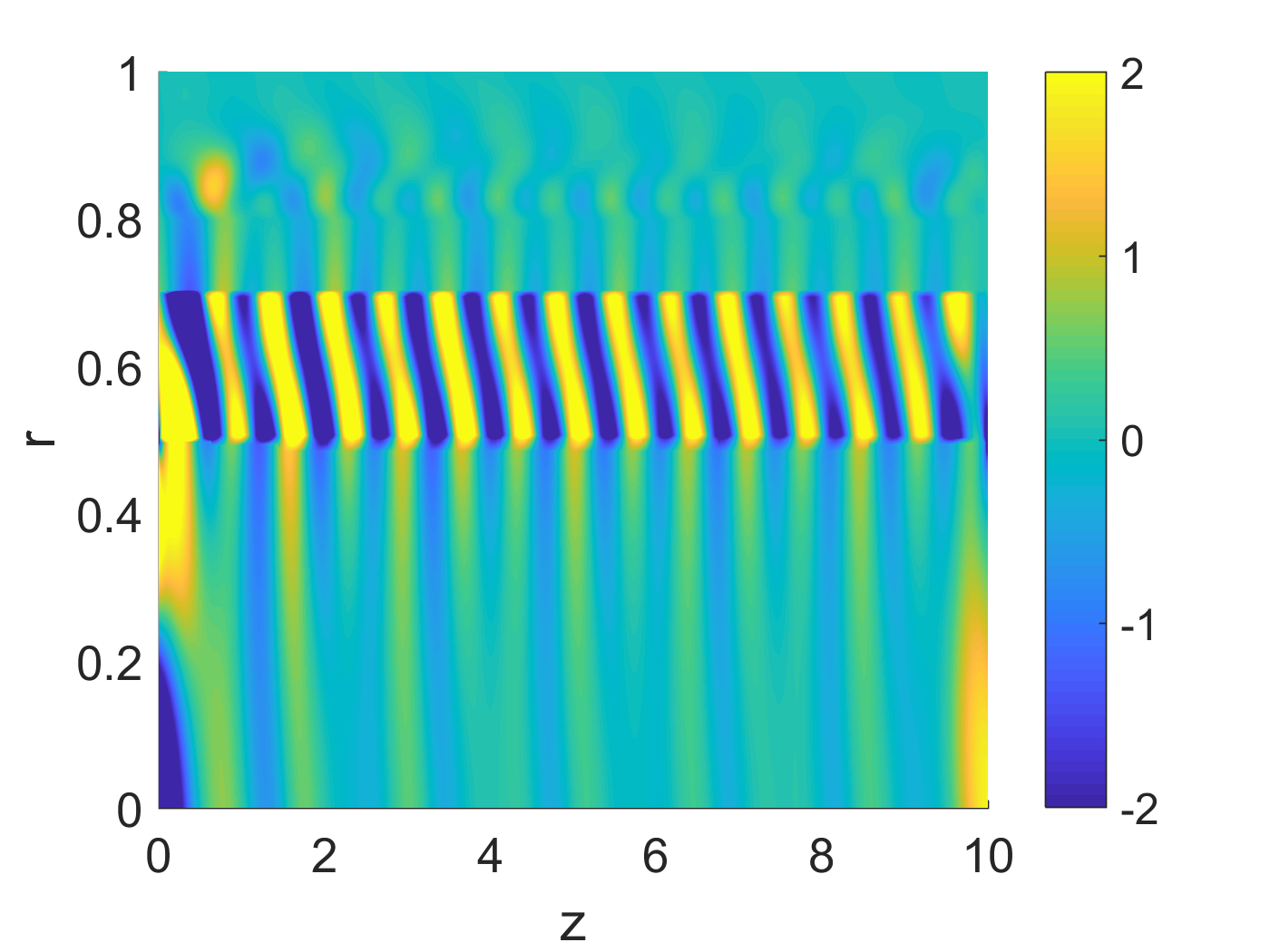} &
\includegraphics[scale=0.1]{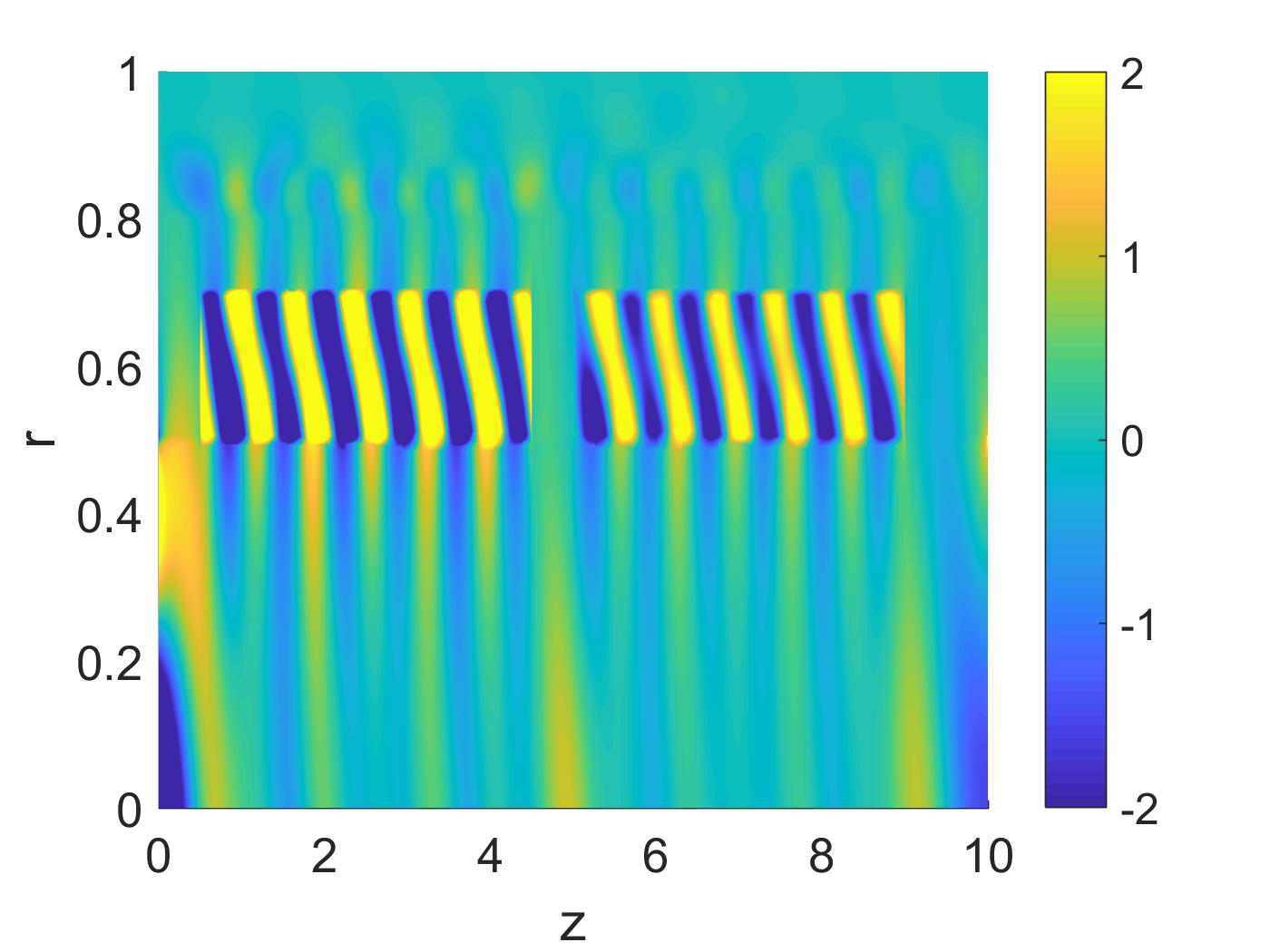} &
\includegraphics[scale=0.1]{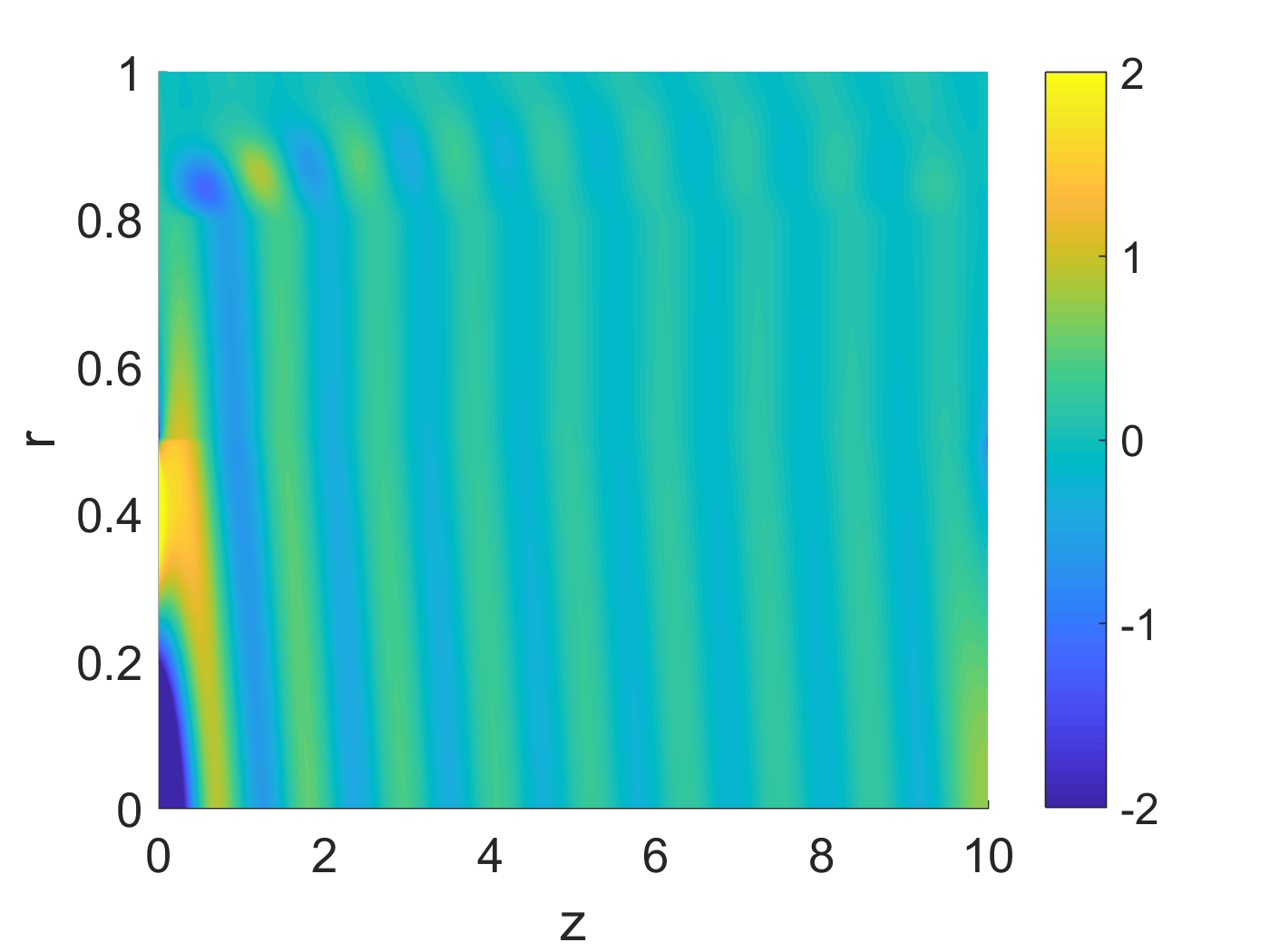}
\end{tabular}
\caption{The real part of the magnetic field $E_z$ for the considered axon structure shown in Figure~\ref{fig:setup}. Left: Figure~\ref{fig:setup}(a), middle: Figure~\ref{fig:setup}(b), right: Figure~\ref{fig:setup}(c).}
\label{fig:Ez}
\end{figure}

\begin{figure}[htbp]
\centering
\begin{tabular}{ccc}
\includegraphics[scale=0.1]{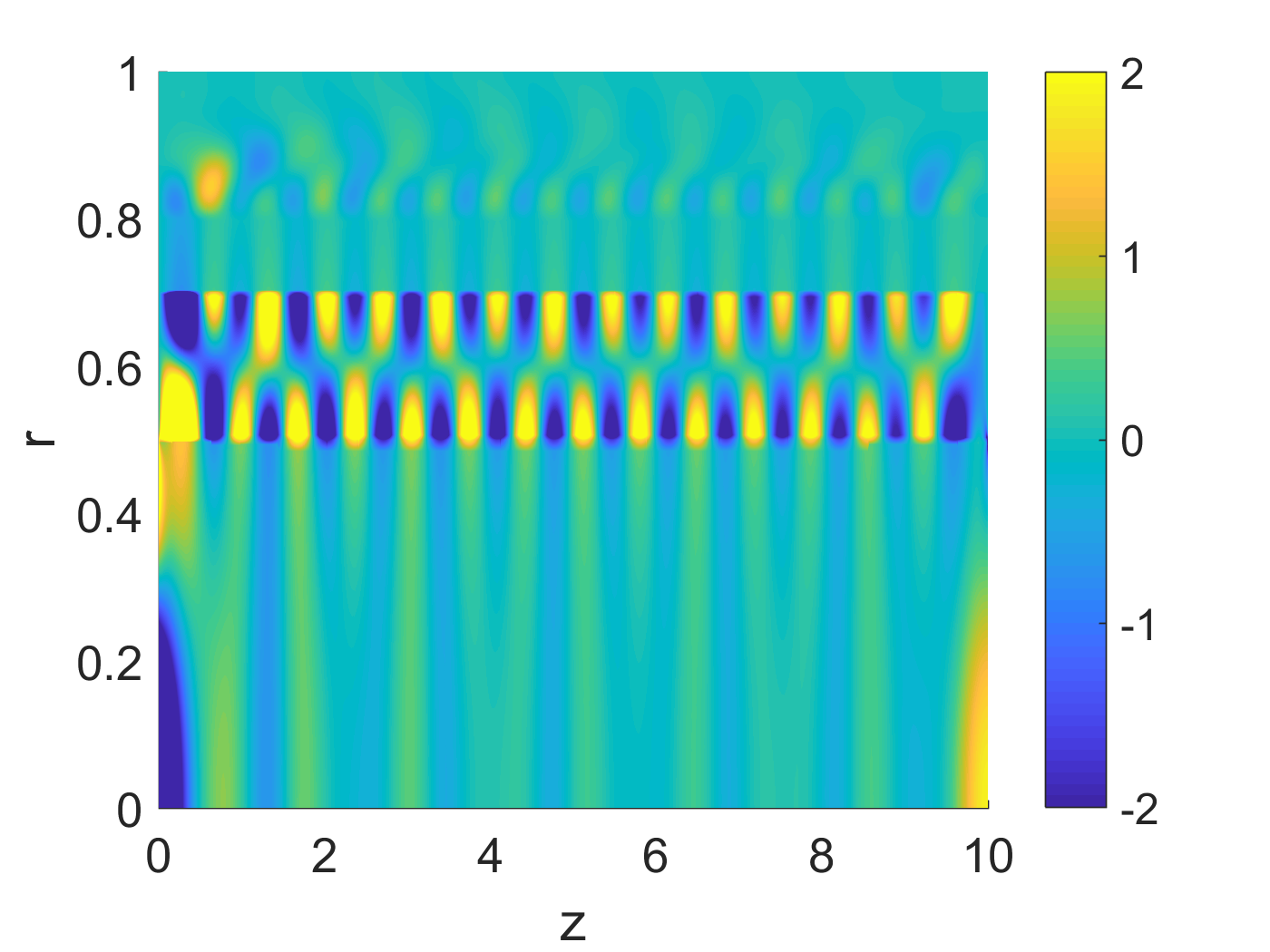} &
\includegraphics[scale=0.1]{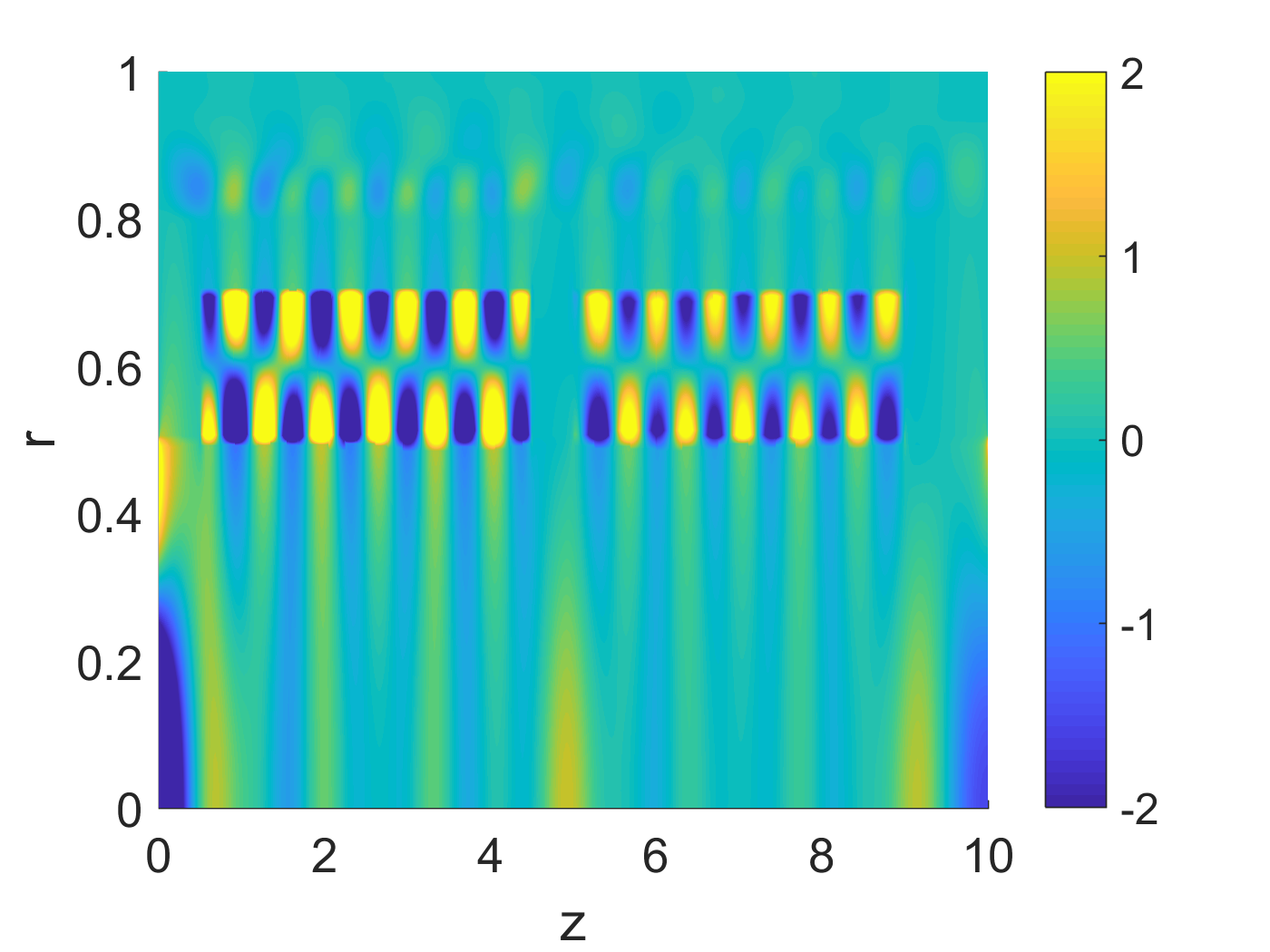} &
\includegraphics[scale=0.1]{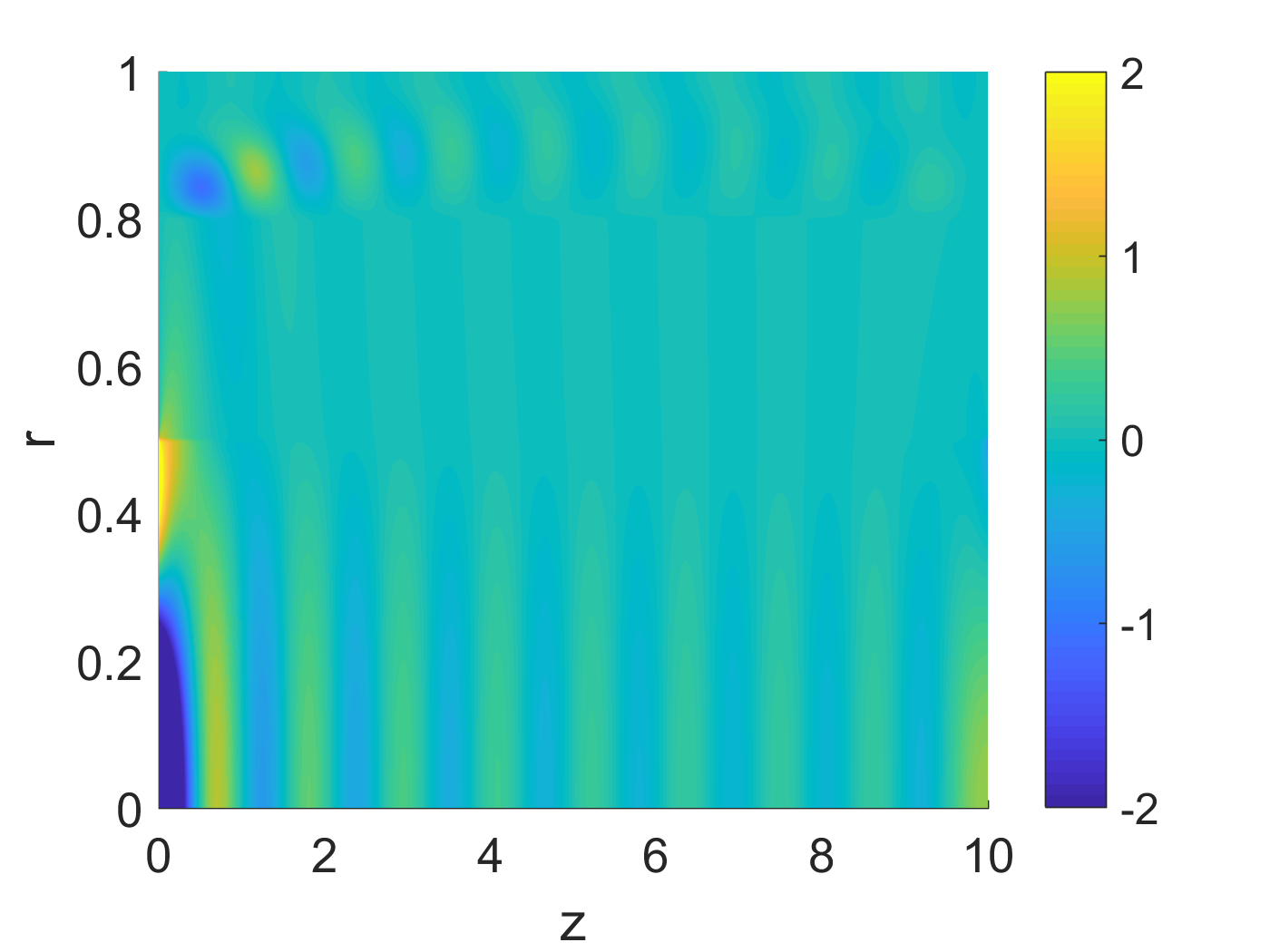}
\end{tabular}
\caption{The real part of the magnetic field $E_r$ for the considered axon structure shown in Figure~\ref{fig:setup}. Left: Figure~\ref{fig:setup}(a), middle: Figure~\ref{fig:setup}(b), right: Figure~\ref{fig:setup}(c).}
\label{fig:Er}
\end{figure}

\section*{Acknowledgement}
XJ was supported in part by the China NSF Grant 12171017. ML was partially supported by the China Postdoctoral Science Foundation 2020TQ0344 and China NSF Grant 12101597. TY gratefully acknowledges support from China NSF Grants 12288201 and 12171465. WZ was supported in part by the China NSF for Distinguished Young Scholars 11725106 and by China NSF major project 11831016.

\end{document}